%% file: boundary_nonunique.tex
\documentclass{amsart}

\theoremstyle{definition}

\usepackage{color}

\input{header}

\begin{document}

\title[Bifurcation for a fully nonlinear boundary Yamabe-type problem]{Nonuniqueness for a fully nonlinear boundary Yamabe-type problem via bifurcation theory}
\author{Jeffrey S. Case}
\thanks{JSC was supported by a grant from the Simons Foundation (Grant No.\ 524601)}
\address{109 McAllister Building \\ Penn State University \\ University Park, PA 16802}
\email{jscase@psu.edu}
\author{Ana Claudia Moreira}
\address{Universidade de S\~ao Paulo \\ Departamento de Matem\'atica \\ Rua do Mat\~ao, 1010 \\ S\~ao Paulo, SP, 05508-090, Brazil}
\email{anaclaudia77br@gmail.com}
\email{amoreira@ime.usp.br}
\author{Yi Wang}
\thanks{YW was partially supported by NSF Grant No.\ DMS-1612015}
\address{3400 N. Charles St. \\ 404 Krieger Hall \\ Mathematics Department \\ Baltimore, MD 21218}
\email{ywang@math.jhu.edu}
\begin{abstract}
 One way to generalize the boundary Yamabe problem posed by Escobar is to ask if a given metric on a compact manifold with boundary can be conformally deformed to have vanishing $\sigma_k$-curvature in the interior and constant $H_k$-curvature on the boundary.  When restricting to the closure of the positive $k$-cone, this is a fully nonlinear degenerate elliptic boundary value problem with fully nonlinear Robin-type boundary condition.  We prove a general bifurcation theorem which allows us to construct examples of compact Riemannian manifolds $(X,g)$ for which this problem admits multiple non-homothetic solutions in the case when $2k<\dim X$.  Our examples are all such that the boundary with its induced metric is a Riemannian product of a round sphere with an Einstein manifold.
\end{abstract}
\maketitle

\input{intro}

\input{extension}
\input{bg}
\input{fredholm}
\input{bifurcation}
\input{example}
\bibliographystyle{abbrv}
\bibliography{boundary_nonunique_bib}
\end{document}

%% file: header.tex
\usepackage{amsmath}
\usepackage{amssymb}
\usepackage{amsthm}
\usepackage{mathrsfs}

\DeclareMathOperator{\sech}{sech}

\DeclareMathOperator{\Id}{Id}

\DeclareMathOperator{\tr}{tr}

\DeclareMathOperator{\Vol}{Vol}
\DeclareMathOperator{\dvol}{dvol}

\DeclareMathOperator{\Ric}{Ric}

\DeclareMathOperator{\Rm}{Rm}

\DeclareMathOperator{\End}{End}

\DeclareMathOperator{\Sym}{Sym}

\DeclareMathOperator{\coker}{coker}
\DeclareMathOperator{\FredholmIndex}{Index}
\DeclareMathOperator{\Ind}{Ind}
\DeclareMathOperator{\Int}{Int}
\DeclareMathOperator{\csch}{csch}

\newcommand{\Cadm}[1]{C_{#1,\mathrm{adm}}}

\newcommand{\oB}{\overline{B}}

\newcommand{\odelta}{\overline{\delta}}

\newcommand{\onabla}{\overline{\nabla}}
\newcommand{\cg}{\widetilde{g}}

\newcommand{\cw}{\widetilde{w}}

\newcommand{\hg}{\widehat{g}}

\newcommand{\lv}{\lvert}
\newcommand{\rv}{\rvert}
\newcommand{\lV}{\lVert}
\newcommand{\rV}{\rVert}

\newcommand{\mD}{\mathcal{D}}
\newcommand{\mE}{\mathcal{E}}
\newcommand{\mF}{\mathcal{F}}

\newcommand{\mH}{\mathcal{H}}

\newcommand{\mV}{\mathcal{V}}

\newcommand{\bN}{\mathbb{N}}

\newcommand{\bR}{\mathbb{R}}

\newcommand{\sC}{\mathscr{C}}

\newcommand{\sT}{\mathscr{T}}
\newcommand{\sU}{\mathscr{U}}
\newcommand{\sV}{\mathscr{V}}
\newcommand{\sW}{\mathscr{W}}

\newcommand{\suchthat}{\mathrel{}\middle|\mathrel{}}

\def\sideremark#1{\ifvmode\leavevmode\fi\vadjust{\vbox to0pt{\vss
 \hbox to 0pt{\hskip\hsize\hskip1em
 \vbox{\hsize3cm\tiny\raggedright\pretolerance10000
 \noindent #1\hfill}\hss}\vbox to8pt{\vfil}\vss}}}

\newcommand{\comment}[1]{}

\newtheorem{thm}{Theorem}[section]
\newtheorem{prop}[thm]{Proposition}
\newtheorem{lem}[thm]{Lemma}
\newtheorem{cor}[thm]{Corollary}

\theoremstyle{definition}
\newtheorem{defn}[thm]{Definition}

\theoremstyle{remark}
\newtheorem{remark}[thm]{Remark}

\numberwithin{equation}{section}

%% file: intro.tex
\section{Introduction}
\label{sec:intro}

An important question for geometric variational problems is whether their solutions are unique.  Even in the case of the Yamabe Problem, where one wants to find a Riemannian metric of constant scalar curvature and unit volume in a given conformal class, this question has a rich and deep history.  For conformal classes which do not admit a metric of positive scalar curvature, there is always a unique solution, while the action of the M\"obius group gives a noncompact set of solutions on the round spheres~\cite{Schoen_Yau_book}.  In dimension at least three, there are a variety of ways to construct nonspherical conformal classes with multiple solutions, such as Schoen's construction~\cite{Schoen1989} on $S^1\times S^{n-1}$ using ODE methods, later generalized to other products by Petean~\cite{Petean2010}, or by bifurcation theory as pioneered by de Lima, Piccione and Zedda~\cite{deLimaPiccioneZedda2012} and later applied to a variety of examples (e.g.\ \cite{BettiolPiccione2013,BettiolPiccioneSantoro2016}).  By constrast, an easy maximum principle argument shows that on closed surfaces not conformal to the sphere, there is a unique solution up to homothety (cf.\ \cite{GurskyStreets2015}).  More generally, Khuri, Marques and Schoen~\cite{KhuriMarquesSchoen2009} proved that on manifolds of dimension at most $24$ on which the Positive Mass Theorem holds, the set of solutions is compact.  By constrast, Brendle~\cite{Brendle2008} and Brendle and Marques~\cite{BrendleMarques2009} gave examples of nonspherical metrics in dimensions at least $25$ for which the set of solutions is noncompact.

Analogues of these results have been considered for many other geometric variational problems; here we mention only a few such results.  On closed manifolds, a natural generalization of the Yamabe Problem is to find a Riemannian metric of constant $\sigma_k$-curvature and unit volume in a given conformal class~\cite{Viaclovsky2000}.  Here the $\sigma_k$-curvature is the $k$-th elementary symmetric polynomial of the eigenvalues of the Schouten tensor $P$; its importance stems from the decomposition of the Riemann curvature tensor
\[ \Rm = W + P \wedge g \]
into the totally trace-free Weyl tensor $W$ and the Kulkarni--Nomizu product $P\wedge g$ of the Schouten tensor and the Riemannian metric.  Since the Weyl tensor is conformally invariant, the Schouten tensor completely controls the Riemann curvature tensor within a conformal class; note also that $\sigma_1$ is proportional to the scalar curvature.  By adapting Schoen's ODE method~\cite{Schoen1989} for the case $k=1$, Viaclovsky~\cite{Viaclovsky2000} showed that if $n>2k$, then $S^1\times S^{n-1}$ admits multiple nonhomothetic solutions to the constant $\sigma_k$-curvature problem provided the $S^1$-factor has sufficiently large radius.  Recent work of Gursky and Streets~\cite{GurskyStreets2016b,GurskyStreets2015,GurskyStreets2016} shows that the condition $n>2k$ is required; when $n=2k$, $k\leq2$, the solutions are only nonunique in the conformal class of the round sphere~\cite{GurskyStreets2015,GurskyStreets2016}, and when $n=2k$, $k\geq3$, the only locally conformally flat manifold for which solutions are nonunique is the sphere~\cite{GurskyStreets2016b}.

On compact manifolds with boundary, one could instead study Escobar's problem~\cite{Escobar1992a,Escobar1992} of conformally deforming a Riemannian metric so that the interior has constant scalar curvature and the boundary has constant mean curvature; one typically asks that one of these constants is zero. More precisely, Escobar first proved that for manifolds with boundary, $(J_g; H_g)$ is variational, where $J_g:=R_g/2n$; variational follows from the fact that the functional
\[ G_1(g) := \frac{1}{n-1}\left(\int_X J_g \dvol_g + \oint_M H_g \dvol_{\iota^\ast g}\right) \]
(note our convention is $H_g=\frac{1}{n}tr A$) for $g\in [g_0]$ is such that 
\[ \left.\frac{d}{dt}\right|_{t=0}G_1\left(e^{2t\Upsilon}g\right) = \int_X R_g\Upsilon\,\dvol_g + \oint_M H_g\Upsilon\,\dvol_{\iota^\ast g} . \]
Thus depending on the different volume constraints, there are two types of boundary Yamabe problem. If one considers, 
\[ \inf_{g\in [g_0]} \frac{\displaystyle\int_X J_g \dvol_g + \oint_M H_g \dvol_{\iota^\ast g}}{\Vol_g (M)^{\frac{n-1}{n}}}, \]
then the Euler-Lagrangian equation (with suitable normalization) is given by
\[
 \begin{cases}
  R_g = 0, & \text{in $X$}, \\
  H_g = \mbox{const}, & \text{on $M$}.
 \end{cases}
\]
This boundary Yamabe problem is called the scalar flat type. This problem, as remarked by Escobar in \cite{Escobar1992a}, is a higher dimensional generalization of the Riemann mapping theorem. It was studied by Escobar~\cite{Escobar1992a,Escobar1996} and Marques~\cite{Marques2005,Marques2007}.  If one considers instead
\[ \inf_{g\in [g_0]} \frac{\displaystyle\int_X J_g \dvol_g + \oint_M H_g \dvol_{\iota^\ast g}}{\Vol_g (X)^{\frac{n-1}{n+1}} }, \]
then the Euler-Lagrangian equation is given by
\[
 \begin{cases}
  R_g = \mbox{const}, & \text{in $X$}, \\
  H_g = 0, & \text{on $M$}.
 \end{cases}
\]
This boundary Yamabe problem is called the minimal boundary type. It was studied by Escobar \cite{Escobar1992}, and recently by Brendle and Chen \cite{BrendleChen2014}.
Moreover, nonuniqueness when the interior dimension is at least three has recently been established using bifurcation theory both 
in the minimal boundary \cite{DiazMoreira2018,Moreira2018} and the scalar flat \cite{Diaz2018} cases.
In the latter case, a direct study of the underlying PDE has been carried out by de la Torre, del Pino, Gonz\'alez and Wei~\cite{DelaTorreDelPinoGonzalezWei2017}.  Here one faces the additional difficulty that one cannot use symmetry to reduce the problem to a local ODE; instead, one either reduces to a nonlocal ODE on the boundary or to a PDE of two variables in the interior.  This interplay between interior operators and nonlocal operators is important to the work of de la Torre \emph{et.\ al.}~\cite{DelaTorreDelPinoGonzalezWei2017}, and in fact allows them to study nonuniqueness for the general fractional Yamabe problem~\cite{GonzalezQing2010}.  Bifurcation methods have also recently been used by Bettiol, Piccione and Sire~\cite{BettiolPiccioneSire2018} to prove nonuniqueness for metrics with constant fourth-order $Q$-curvature.

The goal of this paper is to establish similar nonuniquess results for the $\sigma_k$-curvature on manifolds with boundary.  We begin by describing the appropriate boundary geometry.  In general, the $\sigma_k$-curvature problem is only variational when $k\leq2$ or the underlying conformal class is locally conformally flat~\cite{BransonGover2008,Viaclovsky2000}; this is a key reason for the dichotomy between the cases $k\leq2$ and $k\geq3$ in the aforementioned results of Gursky and Streets~\cite{GurskyStreets2016b,GurskyStreets2015,GurskyStreets2016}.  S.\ Chen~\cite{Chen2009s} showed that there are local scalar invariants $H_k$ on the boundary, constructed as polynomials in the restriction of the interior Schouten tensor and the second fundamental form of the boundary, for which the pair $(\sigma_k;H_k)$ is variational; i.e.\ such that there is a functional $G_k\colon [g_0]\to\bR$, for $[g_0]$ a conformal class on a compact manifold $X$ with boundary $M=\partial X$, such that
\begin{equation}
 \label{eqn:variational}
 \left.\frac{d}{dt}\right|_{t=0}G_k\left(e^{2t\Upsilon}g\right) = \int_X \sigma_k^g\Upsilon\,\dvol_g + \oint_M H_k^g\Upsilon\,\dvol_{\iota^\ast g}
\end{equation}
for all $g\in [g_0]$ and all $\Upsilon\in C^\infty(X)$ (see also~\cite{CaseWang2016s}).  When $\dim X\not=2k$, one can take
\begin{equation}
 \label{eqn:mFk_noncritical_dimension}
 G_k(g) = \frac{1}{n+1-2k}\left[ \int_X \sigma_k^g\,\dvol_g + \oint_M H_k^g\,\dvol_{\iota^\ast g} \right] ,
\end{equation}
where $n=\dim M$.
Note that when $k=1$, $H_k^g=H_g$, the mean curvature of boundary, and $G_k$ is the functional considered by Escobar in the boundary Yamabe problem.

We are interested in studying the $\sigma_k$-flat type of boundary Yamabe problem, generalizing Escobar's work when $k=1$. The set of solutions are the critical points of 
\begin{equation}
\inf_{g\in [g_0], g\in \overline{\Gamma_k^+}} \frac{\displaystyle\int_X \sigma_k^g \dvol_g + \oint_M H_k^g \dvol_{\iota^\ast g}}{\Vol_g (M)^{\frac{n+1-2k}{n}}},
\end{equation}
They satisfy
\begin{equation}
 \label{eqn:constant_Hk_problem}
 \begin{cases}
  \sigma_k^g = 0, & \text{in $X$}, \\
  H_k^g = \mbox{const}, & \text{on $M$}
 \end{cases}
\end{equation}
in a given conformal class $[g_0]$ on $X$, at least in the cases when the $\sigma_k$-curvature is variational.  
More precisely, let
\[ \Gamma_k^+ = \left\{ g\in [g_0] \suchthat \sigma_1^g>0, \dotsc, \sigma_k^g>0 \right\} \]
be the positive elliptic $k$-cone, so named because if $g=e^{2\Upsilon}g_0\in\Gamma_k^+$, then the fully nonlinear PDE $\sigma_k^g=f$ is an elliptic equation in $\Upsilon$ when written with respect to the background metric $g_0$.  We study the space of solutions $g\in\overline{\Gamma_k^+}$ to~\eqref{eqn:constant_Hk_problem}, where the closure is taken in $C^{1,1}(X)$.  Written in terms of a fixed background metric, this problem becomes a fully nonlinear degenerate elliptic PDE with fully nonlinear Robin-type boundary condition. 

Similar to the special case $k=1$ treated by de la Torre \textit{et. al.}~\cite{DelaTorreDelPinoGonzalezWei2017}, attempting to prove nonuniqueness for the boundary value problem~\eqref{eqn:constant_Hk_problem} using Schoen's symmetry argument runs into the problem that one can only reduce the problem to a local PDE of two variables.  Moreover, since $\sigma_k$ is fully nonlinear, it is unclear if one could reduce the problem to a nonlocal ODE on the boundary.  Thus we attack the problem using bifurcation methods.  To that end, it will be helpful to recast~\eqref{eqn:constant_Hk_problem} as a nonlocal problem on the boundary $M$. Suppose $h$ is a Riemannian metric on $M$ for which there exists a metric $g\in\Gamma_k^+$ with $g\rv_{TM}=h$.  Guan~\cite{Guan2007} showed that there is a solution $\cg\in\overline{\Gamma_k^+}$ of the Dirichlet problem
\begin{equation}
 \label{eqn:metric_extension}
 \begin{cases}
  \sigma_k^{\cg} = 0, & \text{in $X$}, \\
  \cg\rv_{TM} = h, & \text{on $M$}.
 \end{cases}
\end{equation}
Indeed, when $2k<\dim X$, the first- and third-named authors~\cite{CaseWang2016s} showed that $\cg$ is the unique minimizer of the functional~\eqref{eqn:mFk_noncritical_dimension} among all metrics in $\overline{\Gamma_k^+}$ which induce $h$ on the boundary.  This enables us to define the \emph{nonlocal $k$-curvature $\mH_k$} of $h$ by $\mH_k^h:=H_k^{\cg}$, so that $\cg$ is a solution to~\eqref{eqn:constant_Hk_problem} if and only if $h$ is a solution to
\begin{equation}
 \label{eqn:constant_mHk_problem}
 \mH_k^h = \mbox{const} .
\end{equation}
Note that $\mH_k^h$ depends only on the choice of metric $h$ on $M$ and the given conformal class $[g_0]$ in the interior.

Our main result is a general bifurcation theorem which applies to families of solutions (not necessarily in the same conformal class) to the nonlocal problem~\eqref{eqn:constant_mHk_problem}.  This result gives conditions on the spectrum of the linearization of~\eqref{eqn:constant_mHk_problem} which are sufficient to guarantee the existence of conformal classes which admit nonunique solutions to~\eqref{eqn:constant_mHk_problem}.  The key condition to check is whether the number of negative eigenvalues of the linearization of~\eqref{eqn:constant_mHk_problem} ``jumps'' as one moves along the given family.  A more precise statement is as follows:

\begin{thm}
 \label{thm:bifurcation}
 Fix $k\in\bN$, $4\leq j\in\bN$, and $\alpha\in(0,1)$.  Let $X^{n+1}$ be a compact manifold with boundary $M^n:=\partial X$.  Let $\{g_s\}_{s\in[a,b]}$ be a smooth one-parameter family of $C^\infty$-metrics on $X$ such that $\sigma_k^{g_s}\equiv0$ and with respect to which $M$ has unit volume and constant $H_k$-curvature for all $s\in[a,b]$.  If $k\geq3$, assume additionally that $g_s$ is locally conformally flat for all $s\in[a,b]$.  Suppose that:
 \begin{enumerate}
  \item for every $s\in[a,b]$, the metric $g_s\in\overline{\Gamma_k^+}$ and there is a metric $\hg_s\in\Gamma_k^+$ conformal to $g_s$ and such that $g_s\rv_{TM}=\hg_s\rv_{TM}$;
  \item for every $s\in[a,b]$, either
  \begin{enumerate}
   \item $T_{k-1}^{g_s}>0$ and $S_{k-1}^{g_s}>0$, or
   \item $k=1$;
  \end{enumerate}
  \item the Jacobi operators $D\mF^{g_a}$ and $D\mF^{g_b}$ are nondegenerate; and
  \item $\Ind\bigl(D\mF^{g_a}\bigr)\not=\Ind\bigl(D\mF^{g_b}\bigr)$.
 \end{enumerate}
 Then there exists a point $s_\ast\in(a,b)$ and a sequence $(s_\ell)_\ell\subset[a,b]$ such that $s_\ell\to s_\ast$ as $\ell\to\infty$ and for each $\ell$, there are nonisometric unit volume $C^{j,\alpha}$-metrics in $[g_{s_\ell}\rv_{TM}]$ with constant $\mH_k$-curvature.
\end{thm}

In the statement of Theorem~\ref{thm:bifurcation}, $T_{k-1}^{g_s}$ is the $(k-1)$-th Newton tensor of the Schouten tensor of $g_s$, the tensor $S_{k-1}^{g_s}$ is a section of $S^2T^\ast M$ defined similarly in relation to $H_k$ (see ~\eqref{eqn:Hk_defn} and \eqref{eqn:Sk_defn}), and the Jacobi operator $D\mF$ is closely related to the second variation of the functional~\eqref{eqn:mFk_noncritical_dimension} (see Corollary~\ref{cor:boundary_dF}).  We emphasize that metrics $g_s$ need not be conformal to each other and the conclusion of Theorem~\ref{thm:bifurcation} is that there is a bifurcation instant $s_\ast$ for the family $\{g_s\}$.  See Section~\ref{sec:bg} for further details.

Theorem~\ref{thm:bifurcation} imposes no assumption on the sign of the constant $H_k$-curvature, though we shall only apply it with positive $H_k$-curvature.  We do not know whether there are families which satisfy the hypotheses of Theorem~\ref{thm:bifurcation} with nonpositive $H_k$-curvature.  We have normalized the volume of the boundary in Theorem~\ref{thm:bifurcation}, rather than the $H_k$-curvature as in~\eqref{eqn:constant_Hk_problem}, for convenience.  Of course, a solution of~\eqref{eqn:constant_Hk_problem} can always be rescaled to a metric of constant $H_k$-curvature with respect to which $M$ has unit volume.


Our application of bifurcation theory is substantially different from previous applications to Yamabe-type problems due to complications related the the degenerate fully nonlinear elliptic PDE $\sigma_k^g=0$.  On the one hand, it is not even clear that solutions to $\sigma_k^g=0$ with Dirichlet boundary conditions always exist; Guan's~\cite{Guan2007} result requires the existence of a subsolution.  On the other hand, one typically expects at best $C^{1,1}$-regularity of solutions. We get around these issues in Theorem~\ref{thm:bifurcation} by explicitly assuming the existence of smooth solutions which have $T_{k-1}>0$ and also have subsolutions and then applying Agmon, Douglis and Nirenberg~\cite{AgmonDouglisNirenberg1959}.  Note that both of these conditions are open conditions, and the former implies that the equation $\sigma_k=0$ is in fact elliptic in the interior.  The additional assumption $S_{k-1}>0$ is also open, and implies that the boundary value problem~\eqref{eqn:constant_Hk_problem} is elliptic in the sense of Agmon, Douglis and Nirenberg~\cite{AgmonDouglisNirenberg1959}. These tools allow us to appeal to the Fiberwise Bifurcation Theorem of de Lima, Piccione and Zedda~\cite{deLimaPiccioneZedda2012}, which in particular allows us to solve the fully nonlinear problem~\eqref{eqn:constant_Hk_problem} through careful study of its linearization.

In Section~\ref{sec:example} we construct three general families of Riemannian manifolds which satisfy the hypotheses of Theorem~\ref{thm:bifurcation}.  In particular, these examples show that if $\dim X>2k$, there are infinitely many conformal classes for which solutions of~\eqref{eqn:constant_mHk_problem} are not unique.  Our examples all have the property that the boundary is a Riemannian product a round sphere and an Einstein manifold (cf.\ \cite{deLimaPiccioneZedda2012,Diaz2018,Petean2010}).  For ease of reading, we summarize these examples from the perspective of the boundary in three results corresponding to when the sign of the Ricci curvature of the second factor is negative, positive, or zero, respectively.

\begin{thm}
 \label{thm:nonuniqueness_neg}
 Fix $k\in\{2,3\}$, let $(S^n,d\theta^2)$ denote the round sphere with a metric of constant sectional curvature $1$, and let $(H^m,g_H)$ denote a compact hyperbolic manifold with constant sectional curvature $-1$.  Suppose that one of the following statements holds:
 \begin{enumerate}
  \item $k=2$, and $n=\frac{\ell(\ell+3)}{2}$, $m=\frac{\ell(\ell+1)}{2}$ for some $\ell\geq2$;
  \item $k=3$, and $n=\frac{\ell(3\ell+5)}{2}$, $m=\frac{\ell(3\ell-1)}{2}$ for some $\ell\geq2$;
  \item $k=3$, and $n=\frac{(\ell+2)(3\ell+1)}{2}$, $m=\frac{\ell(3\ell+1)}{2}$ for some $\ell\geq1$.
 \end{enumerate}
 Then there are infinitely many $s\in\bR_+$ such that, up to rescaling, the product $(S^n\times H^m,d\theta^2\oplus s^2\,g_H)$ is a solution to the nonlocal problem~\eqref{eqn:constant_mHk_problem}, but it is not the unique solution.
\end{thm}

\begin{thm}
 \label{thm:nonuniqueness_pos}
 Fix $k\in\{2,3\}$ and let $(S^n,d\theta_n^2)$ and $(S^m,d\theta_m^2)$ denote round spheres with constant sectional curvature $1$.  Suppose that one of the following statements holds:
 \begin{enumerate}
  \item $k=2$, and $n=\frac{(\ell+1)(\ell+2)}{2}$, $m=\frac{(\ell-1)(\ell+2)}{2}$ for some $\ell\geq3$;
  \item $k=3$, and $n=\frac{(\ell+1)(3\ell+2)}{2}$, $m=\frac{(\ell-1)(3\ell+2)}{2}$ for some $\ell\geq3$;
  \item $k=3$, and $n=\frac{(\ell+1)(3\ell+4)}{2}$, $m=\frac{(\ell+1)(3\ell-2)}{2}$ for some $\ell\geq2$.
 \end{enumerate}
 Then there are infinitely many $s\in\bR_+$ such that, up to rescaling, the product $(S^n\times S^m,d\theta_n^2\oplus s^2\,d\theta_m^2)$ is a solution to the nonlocal problem~\eqref{eqn:constant_mHk_problem}, but it is not the unique solution.
\end{thm}

\begin{thm}
 \label{thm:nonuniqueness_zero}
 Let $(S^n,d\theta^2)$ denote the round sphere with a metric of constant sectional curvature $1$ and let $(F^{m-1},g_F)$ be a compact Ricci flat manifold.  Suppose additionally that $n=\frac{(\ell+1)(\ell+2)}{2}$ and $m=\frac{(\ell-1)(\ell+2)}{2}$ for some $\ell\geq2$.  Then there are infinitely many $s_1,s_2\in\bR_+$ such that, up to rescaling, the product
 \[ (S^n\times S^1\times F^{m-1},d\theta^2\oplus s_1^2dt^2\oplus s_2^2g_F) \]
 is a solution to the nonlocal problem~\eqref{eqn:constant_mHk_problem}, but it is not the unique solution.
\end{thm}

As noted above, the nonlocal invariant $\mH_k$ depends on the choice of interior conformal class.  In proving each of the previous three theorems, we use the fact that for a given $k\in\bN$, there are infinitely many pairs $(m,n)\in\bN^2$ such that whenever $(M^m,g_M)$ and $(N^n,g_N)$ are Einstein manifolds with $\Ric_{g_M}=(m-1)g_M$ and $\Ric_{g_N}=-(n-1)g_N$, respectively, their Riemannian product $(M\times N,g:=g_M\oplus g_N)$ has $\sigma_k^g\equiv0$, $T_{k-1}^g>0$, and $g\in\Gamma_{k-1}^+$; see Lemma~\ref{lem:einstein_products}.  This choice of normalization also ensures that in Theorem~\ref{thm:nonuniqueness_neg} and Theorem~\ref{thm:nonuniqueness_pos}, the product metric is locally conformally flat.
This is because a Riemannian product is locally conformally flat if both factors are locally conformally flat and either (i) one of the factors is one-dimensional or (ii) the factors have constant sectional curvature of opposite sign and equal magnitude.
One can prove an analogue of Lemma~\ref{lem:einstein_products} for more general normalizations of the factors and use it to weaken the dimensional constraints for Theorem~\ref{thm:nonuniqueness_zero} and each of Theorem~\ref{thm:nonuniqueness_neg} and Theorem~\ref{thm:nonuniqueness_pos} in the case $k=2$; see Remark~\ref{rk:genl_normalization} for further discussion.

The proofs of Theorem~\ref{thm:nonuniqueness_neg}, Theorem~\ref{thm:nonuniqueness_pos}, and Theorem~\ref{thm:nonuniqueness_zero} are easily modified to include the case $k=1$.  We omit the details as a similar construction has already been given by Diaz~\cite{Diaz2018}.  We expect that these theorems can be generalized to include bifurcation results for the $\mH_k$-curvature for all $k\in\bN$.

This article is organized as follows:

In Section~\ref{sec:extension}, we describe the Banach manifolds on which we work.  This includes a new existence and stability result for solutions of the Dirichlet problem~\eqref{eqn:metric_extension}.

In Section~\ref{sec:bg} we recall some important definitions and facts about the $H_k$-curvature.  We also show that the linearizations of~\eqref{eqn:constant_Hk_problem} and~\eqref{eqn:constant_mHk_problem} are both formally self-adjoint.

In Section~\ref{sec:fredholm} we show that the linearization of the nonlocal problem~\eqref{eqn:constant_mHk_problem} is Fredholm when restricted to appropriate domains and codomains from Section~\ref{sec:extension}.  See Theorem~\ref{thm:fredholm} for a precise statement.

In Section~\ref{sec:bifurcation} we prove Theorem~\ref{thm:bifurcation}.  The key point is that the properties of the linearization of~\eqref{eqn:constant_mHk_problem} established in Section~\ref{sec:fredholm} allow us to apply the general bifurcation theorem of de Lima, Piccione and Zedda~\cite{deLimaPiccioneZedda2012}.

In Section~\ref{sec:example} we study a handful of explicit families of smooth solutions of the boundary value problem~\eqref{eqn:constant_Hk_problem}, and in the process prove Theorem~\ref{thm:nonuniqueness_neg}, Theorem~\ref{thm:nonuniqueness_pos}, and Theorem~\ref{thm:nonuniqueness_zero}.

%% file: extension.tex
\section{Function spaces and Dirichlet problems}
\label{sec:extension}

As discussed in the introduction, the proof of Theorem~\ref{thm:bifurcation} is simplified by restricting our attention to function spaces defined on the boundary of a compact Riemannian manifold.  To return to Theorem~\ref{thm:bifurcation} and the interior problem~\eqref{eqn:constant_Hk_problem}, we need means to extend elements of these function spaces to the interior of the manifold.  This will be done via the $\sigma_k$-curvature.  To that end, we begin by recalling the definition of the $\sigma_k$-curvature and its essential properties.

Given $k\in\bN$, the \emph{$k$-th elementary symmetric function} of a symmetric $d\times d$-matrix $B\in\Sym_d$ is
\[ \sigma_k(B) := \sum_{i_1<\dotso<i_k} \lambda_{i_1}\dotsm\lambda_{i_k}, \]
where $\lambda_1,\dotsc,\lambda_d$ are the eigenvalues of $B$.  One can compute $\sigma_k(B)$ without knowledge of the eigenvalues of $B$ via the formula
\begin{equation}
 \label{eqn:sigma_kronecker}
 \sigma_k(B) = \frac{1}{k!}\delta_{i_1\dotso i_k}^{j_1\dotso j_k}B_{j_1}^{i_1}\dotsm B_{j_k}^{i_k},
\end{equation}
where $\delta_{i_1\dotso i_k}^{j_1\dotso j_k}$ denotes the generalized Kronecker delta,
\[ \delta_{i_1\dotso i_k}^{j_1\dotso j_k} :=
    \begin{cases}
     1, & \text{if $(i_1,\dotsc,i_k)$ is an even permutation of $(j_1,\dotsc,j_k)$,} \\
     -1, & \text{if $(i_1,\dotsc,i_k)$ is an odd permutation of $(j_1,\dotsc,j_k)$,} \\
     0, & \text{otherwise},
    \end{cases}
\]
and Einstein summation convention is employed.  The \emph{$k$-th Newton tensor} of $B$ is the matrix $T_k(B)\in\Sym_d$ with components
\begin{equation}
 \label{eqn:newton_kronecker}
 T_k(B)_i^j := \frac{1}{k!}\delta_{ii_1\dotso i_k}^{jj_1\dotso j_k} B_{j_1}^{i_1}\dotsm B_{j_k}^{i_k} .
\end{equation}
It is clear from~\eqref{eqn:sigma_kronecker} and~\eqref{eqn:newton_kronecker} that $\sigma_k(B)$ and $T_k(B)$ are homogeneous polynomials of degree $k$ in $B$, and hence can both be polarized.  We require the mixed symmetric functions and Newton tensors obtained by inputting two matrices with a given multiplicity into these polarizations.  More precisely, given nonnegative integers $k,\ell$ with $k\geq\ell$ and matrices $B,C\in\Sym_d$, we define
\begin{align*}
 \sigma_{k,\ell}(B,C) & := \frac{1}{k!}\delta_{i_1\dotso i_k}^{j_1\dotso j_k} B_{j_1}^{i_1}\dotsm B_{j_\ell}^{i_\ell} C_{j_{\ell+1}}^{i_{\ell+1}} \dotsm C_{j_k}^{i_k}, \\
 T_{k,\ell}(B,C)_i^j & := \frac{1}{k!}\delta_{ii_1\dotso i_k}^{jj_1\dotso j_k} B_{j_1}^{i_1}\dotsm B_{j_\ell}^{i_\ell} C_{j_{\ell+1}}^{i_{\ell+1}} \dotsm C_{j_k}^{i_k} .
\end{align*}
That is, $\sigma_{k,\ell}(B,C)$ (resp.\ $T_{k,\ell}(B,C)$) is the polarization of $\sigma_k$ (resp.\ $T_k$) evaluated at $\ell$ factors of $B$ and $k-\ell$ factors of $C$.

When considering the $k$-th elementary symmetric function $\sigma_k$, we usually restrict our attention to the \emph{positive $k$-cone}
\[ \Gamma_k^+ := \left\{ B\in\Sym_n \suchthat \sigma_1(B),\dotsc,\sigma_k(B)>0 \right\} \]
and its closure
\[ \overline{\Gamma_k^+} := \left\{ B\in\Sym_n \suchthat \sigma_1(B),\dotsc,\sigma_k(B)\geq 0 \right\} . \]
The primary reasons for this are that $T_{k-1}(B)$ is positive definite (resp.\ nonnegative) for all $B\in\Gamma_k^+$ (resp.\ all $B\in\overline{\Gamma_k^+}$) and that $\Gamma_k^+$ and $\overline{\Gamma_k^+}$ are convex~\cite{CaffarelliNirenbergSpruck1985}.

Let $(X^{n+1},g)$ be a Riemannian manifold.  The \emph{Schouten tensor} $P$ of $g$ is the section
\[ P := \frac{1}{n-1}\left(\Ric - \frac{R}{2n}g\right) \]
of $S^2T^\ast X$, where $\Ric$ is the Ricci tensor and $R:=\tr_g\Ric$ is the scalar curvature of $g$.  We denote by $g^{-1}$ the musical isomorphism mapping $T^\ast X$ to $TX$ and its extension to tensor bundles.  For example, $g^{-1}P$ is the section of $\End(TX)$ defined by
\[ g\left((g^{-1}P)(Y),Z\right) = P(Y,Z) \]
for all sections $Y,Z$ of $TX$.  The \emph{$\sigma_k$-curvature} of $(X,g)$ is
\[ \sigma_k^g := \sigma_k\left(g^{-1}P\right) . \]
and the \emph{$k$-th Newton tensor} is
\[ T_k^g := T_k\left(g^{-1}P\right) . \]
For example, $\sigma_1^g=\frac{1}{2n}R$ is a multiple of the scalar curvature.  When the metric is clear by context, we omit the superscript $g$.  We write $g\in\Gamma_k^+$ (resp.\ $g\in\overline{\Gamma_k^+}$) if for all points $p\in M$, the symmetric matrix representing $(g^{-1}P)_p\in\End(T_pM)$ lies in $\Gamma_k^+$ (resp. $\overline{\Gamma_k^+})$.

Viaclovsky~\cite{Viaclovsky2000} computed the linearization of the $\sigma_k$-curvature within a conformal class (locally conformally flat if $k=3$).  His result can be restated as follows:

\begin{lem}
 \label{lem:conformal_variation_interior}
 Let $(X^{n+1},g)$ be a Riemannian manifold and let $k\in\bN$.  If $k\geq 3$, assume additionally that $g$ is locally conformally flat.  For any $\Upsilon\in C^\infty(X)$, it holds that
 \begin{equation}
  \label{eqn:genl_trans_sigmak}
  \left.\frac{\partial}{\partial t}\right|_{t=0} \sigma_k^{e^{2t\Upsilon}g} = -2k\Upsilon\sigma_k^g - \delta\left(T_{k-1}^g(\nabla\Upsilon)\right) ,
 \end{equation}
 where $\delta=\tr_g\nabla$ denotes the divergence on $(X,g)$.
\end{lem}

It follows that, regarded as a PDE in a conformal class, the equation $\sigma_k^g=f$ is second-order.  Moreover, it is elliptic (resp.\ degenerate elliptic) if and only if $T_{k-1}^g$ is positive or negative definite (resp.\ positive or negative semi-definite).  In particular, the equation $\sigma_k^g=0$ is degenerate elliptic within the cone $\overline{\Gamma_k^+}$.

We now describe the function spaces in which we work and the manner in which we extend their elements to the interior.  To that end, fix a compact Riemannian manifold with boundary $(X^{n+1},g)$ and let $(M^n,h)$ denote the boundary; i.e.\ $M:=\partial X$ and $h:=g\rv_{TM}$.  Fix also $j\in\bN$ and $\alpha\in(0,1)$.  Given $w\in C^{j,\alpha}(M)$, we denote
\begin{equation}
 \label{eqn:mT_defn}
 \sT_w^{j,\alpha} := \left\{ \phi\in C^{j,\alpha}(M) \suchthat \oint_M \phi\,\dvol_{h_w} = 0 \right\} ,
\end{equation}
where $\dvol_{h_w}$ denotes the Riemannian volume element of the metric $h_w:=e^{2w}h$ on $M$.  The space $\sT_1^{j,\alpha}$ corresponding to the choice $w=1$ is of particular interest.  We extend elements of $\sT_1^{j,\alpha}$ to $X$ by solving the boundary value problem
\begin{equation}
 \label{eqn:linearized_extension}
 \begin{cases}
  \delta\left(T_{k-1}(\nabla v)\right) = 0, & \text{in $X$}, \\
  v = \phi, & \text{on $M$} .
 \end{cases}
\end{equation}
Here $\delta, T_{k-1}, \nabla$ are all with respect to the metric $g$.
The important properties of this extension which we require are contained in the following lemma:

\begin{lem}
 \label{lem:embedding}
 Let $(X^{n+1},g)$ be a compact Riemannian manifold with boundary $M^n:=\partial X$ and let $j\in\bN$ and $\alpha\in(0,1)$.  Suppose that $T_{k-1}>0$.  Then for every $\phi\in C^{j,\alpha}(M)$, there is a unique solution $v_\phi\in C^{j+2,\alpha}\left(\Int(X)\right)\cap C^{j,\alpha}(X)$ of~\eqref{eqn:linearized_extension}, where $\Int(X)$ denotes the interior of $X$.
\end{lem}

\begin{proof}
 Since $T_{k-1}>0$, this is a standard elliptic Dirichlet boundary value problem to which one can apply standard theory (e.g.\ \cite[Theorem~6.2]{GilbargTrudinger2001}).
\end{proof}

We also want to extend functions via the Dirichlet problem
\begin{equation}
 \label{eqn:metric_extension_interior}
 \begin{cases}
  \sigma_k^{g_{\cw}}\equiv0, & \text{in $X$}, \\
  g_{\cw} \in \overline{\Gamma_k^+}, \\
  \cw = w, & \text{on $M$} ,
 \end{cases}
\end{equation}
where $g_{\cw}:=e^{2\cw}g$.  This problem is equivalent to~\eqref{eqn:metric_extension}.  Guan~\cite{Guan2007} showed that if $w\in C^{4,\alpha}(M)$ is such that there is a smooth metric $\cg\in\Gamma_k^+$ with $\cg\rv_{TM}=e^{2w}g\rv_{TM}$, then there is a solution $\cw\in C^{1,1}(X)$ of~\eqref{eqn:metric_extension_interior}.  The first- and last-named authors~\cite{CaseWang2016s} showed that this solution is unique.  We require a version of this result for which the extension has better regularity.  This is done by working with a more restrictive class of functions on $M$.  First, for convenience, we introduce the following terminology.

\begin{defn}
 Let $(X,g)$ be a compact manifold with boundary $M$, let $j,k\in\bN$, $j\geq4$, and let $\alpha\in(0,1)$.  A function $w\in C^{j,\alpha}(M)$ is \emph{$k$-admissible} if there is a function $\cw\in C^{j,\alpha}(X)$ such that $\cw\rv_M=w$ and $g_{\cw}\in\Gamma_k^+$.  We denote
 \[ \Cadm{k}^{j,\alpha}(M) := \left\{ w\in C^{j,\alpha}(M) \suchthat \text{$w$ is $k$-admissible} \right\} . \]
\end{defn}

The above discussion implies that if $w\in\Cadm{k}^{j,\alpha}(M)$, $j\geq4$, then there is a unique extension $\cw\in C^{1,1}(X)$ satisfying~\eqref{eqn:metric_extension_interior}.  We denote by $\cg_w:=e^{2\cw}g$ the metric determined by this extension.  The distinction between ${\cg}_w$ and $g_{\cw}$ is that $w$ is only defined on the boundary, while $\cw$ is defined in the interior.  So the subscript $w$ here is emphasizing that we have a metric determined only by data on the boundary.

To get extensions with improved regularity, we introduce the spaces
\begin{equation}
 \label{eqn:mV_defn}
 \sV_k^{j,\alpha} := \left\{ w\in\Cadm{k}^{j,\alpha}(M) \suchthat \cg_w\in\Gamma_{k-1}^+, T_{k-1}^{\cg_w}>0, \oint_M \dvol_{h_w}=1 \right\} .
\end{equation}
The normalization of the volume is made to avoid the homothety invariance of~\eqref{eqn:constant_mHk_problem}.  The key point here is that the assumption $T_{k-1}^{\cg_w}>0$ allows us to apply the Implicit Function Theorem to conclude that $\sV_k^{j,\alpha}$ is a Banach manifold and that the extensions of elements of $\sV_k^{j,\alpha}$ have improved regularity.

\begin{prop}
 \label{prop:mV_banach}
 Let $(X,g)$ be a compact manifold with boundary $M$, let $j,k\in\bN$ with $j\geq4$, and let $\alpha\in(0,1)$.  Then:
 \begin{enumerate}
  \item for every $w\in\sV_k^{j,\alpha}(M)$, the extension $\cw$ by~\eqref{eqn:metric_extension_interior} is in $C^{j,\alpha}(X)$;
  \item $\sV_k^{j,\alpha}$ is a Banach manifold, and for every $w\in\sV_k^{j,\alpha}$, the tangent space $T_w\sV_k^{j,\alpha}$ is isomorphic to $\sT_w^{j,\alpha}$.
 \end{enumerate}
\end{prop}

\begin{proof}
(1) is straightforward by standard elliptic regularity when $T_{k-1}^{\cg_w}>0$. \\
(2) Since $W\mapsto \sigma_k^{g_W}$, $g_W:=e^{2W}g$, is a $C^2$-map, it readily follows that $\Cadm{k}^{j,\alpha}(M)$ is an open subset of $C^{j,\alpha}(M)$.

 Define
 \[ \Psi\colon C^{j,\alpha}(X)\times \Cadm{k}^{j,\alpha}(M) \to C^{j-2,\alpha}(X) \times C^{j,\alpha}(M) \]
 by
 \begin{equation}
  \label{eqn:Psi}
  \Psi(W,w) := \left( \sigma_k^{g_W}, W\rv_M - w\right),
 \end{equation}
 where $g_W:=e^{2W}g$ and define
 \[ \sW_k^{j,\alpha} := \left\{ w\in \Cadm{k}^{j,\alpha}(M) \suchthat \cg_w\in\Gamma_{k-1}^+, T_{k-1}^{\cg_w} > 0 \right\} . \]
 Suppose that $(W_0,w_0)\in\Psi^{-1}\bigl((0,0)\bigr)$ is such that $g_{W_0}\in\Gamma_{k-1}^+$ and $T_{k-1}^{g_{W_0}}>0$.  By the uniqueness of solutions of~\eqref{eqn:metric_extension_interior}, $W_0=\cw_0$, and hence $w_0\in\sW_k^{j,\alpha}$.  Moreover, the linearization $D^1\Psi_{(W_0,w_0)}\colon C^{j,\alpha}(X) \to C^{j-2,\alpha}(X) \times C^{j,\alpha}(M)$ in the first component is
 \[ D^1\Psi_{(W_0,w_0)}(V) = \left( -\delta\left(T_{k-1}(\nabla V)\right), V\rv_M \right) , \]
 where all geometric quantities are computed with respect to $g_{W_0}$.  In particular, given $(\Phi,\phi)\in C^{j-2,\alpha}(X)\times C^{j,\alpha}(M)$, it holds that $D^1\Psi_{(W_0,w_0)}(V)=(\Phi,\phi)$ if and only if
 \begin{equation}
  \label{eqn:DPsi}
  \begin{cases}
   -\delta\left(T_{k-1}(\nabla V)\right) = \Phi, & \text{in $X$}, \\
   V\rv_M = \phi, & \text{on $M$}.
  \end{cases}
 \end{equation}
 Since $T_{k-1}>0$, standard elliptic theory applied to~\eqref{eqn:DPsi} implies that $D^1\Psi_{(W_0,w_0)}$ is bijective.  The Implicit Function Theorem (e.g.\ \cite[Theorem~17.6]{GilbargTrudinger2001}) then implies that there is a neighborhood $\sU_0\subset \Cadm{k}^{j,\alpha}(M)$ of $w_0$ and a continuous map $\mE\colon\sU_0\to C^{j,\alpha}(X)$ such that $\mE(w_0)=W_0$ and
 \[ \Psi\left(\mE(w),w\right) = (0,0) \]
 for all $w\in\sU_0$.  In particular, $\hg_w:=e^{2\mE(w)}g$ satisfies $\sigma_k^{\hg_w}\equiv0$.  Using the openness in $C^{j,\alpha}(X)$ of the conditions $T_{k-1}^{g_W}>0$ and $g_W\in\Gamma_{k-1}^+$, we may also conclude, by shrinking $\sU_0$ if necessary, that $\hg_w\in\Gamma_{k-1}^+$ and $T_{k-1}^{\hg_w}>0$ for all $w\in\sU_0$.  Thus, by the uniqueness of solutions of~\eqref{eqn:metric_extension_interior}, $\sU_0\subset\sW_k^{j,\alpha}$ and $T_w\sW_k^{j,\alpha}$ is isomorphic to $C^{j,\alpha}(M)$.  Finally, since the volume map $\mV\colon\sW_k^{j,\alpha}\to\bR$ given by
 \[ \mV(w) := \oint_M \dvol_{h_w} \]
 is a submersion and since $\sV_k^{j,\alpha}=\mV^{-1}(1)$, we readily conclude that $\sV_k^{j,\alpha}$ is a Banach manifold and $T_w\sV_k^{j,\alpha}$ is isometric to $\sT_w^{j,\alpha}$ for all $w\in\sV_k^{j,\alpha}$.

\end{proof}

\begin{remark}
 It would be interesting to know if for any $w\in \Cadm{k}^{2,\alpha}$, there is a unique solution $\cw\in C^{1,1}(X)$ of~\eqref{eqn:metric_extension_interior}.  If so, then we could define $\sV^{2,\alpha}$ without requiring $T_{k-1}^{g_{\cw}}>0$ and readily extend Theorem~\ref{thm:bifurcation} to construct conformal classes containing nonhomothetic $C^{1,1}$ solutions of~\eqref{eqn:constant_Hk_problem}.
\end{remark}

%% file: bg.tex
\section{The $H_k$-curvature and formal properties}
\label{sec:bg}

We now turn to describing the $H_k$-curvature and formal properties related to the pair $(\sigma_k;H_k)$.  To that end, let $(X^{n+1},g)$ be a compact Riemannian manifold with boundary $M^n=\partial X$.  Suppose further that $M$ has unit volume with respect to the metric $h:=\iota^\ast g$ induced by $M$, where $\iota\colon M\to X$ is the inclusion map.  Denote by $\eta$ the outward-pointing unit normal vector field with respect to $g$ along $M$.  The \emph{second fundamental form $A$} of $M$ is the section of $S^2T^\ast M$ defined by
\[ A(Y,Z) := g\left(\nabla_Y\eta,Z\right) \]
for all sections $Y,Z$ of $TM$.  We denote by $h^{-1}$ the musical isomorphism mapping $T^\ast M$ to $TM$ and its extension to tensor bundles.  The \emph{$H_k$-curvature} of $M$ is
\begin{equation}\label{eqn:Hk_defn} H_k^g := \sum_{j=0}^{k-1} \frac{(2k-j-1)!(n+1-2k+j)!}{j!(n+1-k)!(2k-2j-1)!!}\sigma_{2k-j-1,j}\left(h^{-1}\iota^\ast P, h^{-1}A\right) . \end{equation}
For example, $H_1^g=\frac{1}{n}\tr_hA$ is the mean curvature of $M$.

Escobar~\cite{Escobar1988} showed that the pair $(\sigma_1;H_1)$ is variational on manifolds with boundary, and S.\ Chen~\cite{Chen2009s} introduced the $H_k$-curvatures so that the same is true of the pair $(\sigma_k;H_k)$ when $k\leq2$ or $g$ is locally conformally flat.  This fact is an immediate consequence of Lemma~\ref{lem:conformal_variation_interior} and the conformal linearization of the $H_k$-curvature.

\begin{lem}[{see~\cite[Lemma~2.2]{CaseWang2016s}}]
 \label{lem:conformal_variation}
 Let $(X^{n+1},g)$ be a Riemannian manifold with boundary $M=\partial X$ and let $k\in\bN$.  If $k\geq 3$, assume additionally that $g$ is locally conformally flat.  For any $\Upsilon\in C^\infty(X)$, it holds that
 \begin{equation}
  \label{eqn:genl_trans_Hk}
  \left.\frac{\partial}{\partial t}\right|_{t=0} H_k^{e^{2t\Upsilon}g} = -(2k-1)\Upsilon H_k^g + T_{k-1}(\eta,\nabla\Upsilon) - \odelta\left(S_{k-1}(\onabla\Upsilon)\right),
 \end{equation}
 where $\onabla$ and $\odelta=\tr_{h}\onabla$ denote the Levi-Civita connection and divergence of $h$, respectively, and
 \begin{equation}
  \label{eqn:Sk_defn}
  S_{k-1} := \sum_{j=0}^{k-2} \frac{(2k-j-3)!(n+2-2k+j)!}{j!(n+1-k)!(2k-2j-3)!!}T_{2k-j-3,j}\left(h^{-1}\iota^\ast P, h^{-1}A\right) ,
 \end{equation}
 with the convention that the empty summation equals zero.
\end{lem}

Note that the right-hand side of~\eqref{eqn:genl_trans_Hk} depends only on the horizontal two-jet and the full one-jet of $\Upsilon$, and hence Lemma~\ref{lem:conformal_variation} makes sense when $\Upsilon\in C^1(X)\cap C^2(M)$.  Moreover, when $k=1$, the right-hand side of~\eqref{eqn:genl_trans_Hk} makes sense when $\Upsilon\in C^1(X)$.

In this article we study metrics $g$ such that $\sigma_k^g\equiv0$ and $H_k^g$ is constant.  When $n+1\not=2k$, such metrics can be characterized as critical points of the functional~\eqref{eqn:mFk_noncritical_dimension} within the set of conformal metrics of unit boundary volume~\cite{CaseWang2016s}.  Here we find it more useful to characterize such metrics in terms of the function
\[ F^g \colon \sC^{j,\alpha} \to C^{j-2,\alpha}(X) \times C^{j-2,\alpha}(M) \]
given by
\begin{equation}
 \label{eqn:defn_F}
 F^g(w) := \left( \sigma_k^{g_w}, H_k^{g_w} - \oint_M H_k^{g_w}\,\dvol_{\iota^\ast g}\right),
\end{equation}
where $g_w:=e^{2w}g$ and
\[ \sC^{j,\alpha} := \left\{ w\in C^{j,\alpha}(X) \suchthat \oint_M \dvol_{h_w} = 1 \right\} \]
denotes the set of conformal factors which induce unit volume metrics on the boundary.  We emphasize that the integration in~\eqref{eqn:defn_F} is taken with respect to the background metric $\iota^\ast g=h$.  Thus $F^g$ is the gradient of the functional
\[ \sC^{j,\alpha} \ni w \mapsto G_k(g_w) \]
at $w=1$ under the volume constraint $\Vol_{h_w}(M)=1$, where $G_k$ is the functional~\eqref{eqn:mFk_noncritical_dimension}. A key point is that $\sigma_k^{g_w}\equiv0$ and $H_k^{g_w}$ is constant if and only if $w\in F^{-1}\left((0,0)\right)$.  We drop the superscript $g$ when the background metric is clear by context.

\begin{remark}
 \label{rk:mapping_properties}
 When $k=1$, the function $F$ in fact takes values in $C^{j-2,\alpha}(X)\times C^{j-1,\alpha}(M)$.  In order to make our treatment of the cases $k\in\bN$ more uniform, we ignore this gain in regularity for the remainder of this section, and only come back to it in Section~\ref{sec:fredholm}.
\end{remark}

Our main result, Theorem~\ref{thm:bifurcation}, is based on the properties of the linearization of $F$ at a point $w\in F^{-1}\left((0,0)\right)$.  To that end, note that the tangent space to $\sC^{j,\alpha}$ at $w\in\sC^{j,\alpha}$ is
\[ T_w\sC^{j,\alpha} = \left\{ v \in C^{j,\alpha}(X) \suchthat \oint_M \iota^\ast v\,\dvol_{h_w} = 0 \right\} . \]
As in our description of the relation between $G_k$ and $F$, we only require the linearization of $F$ at $w=1$.

\begin{prop}
 \label{prop:interior_dF}
 Let $(X^{n+1},g)$ be a compact Riemannian manifold with boundary $M^n:=\partial X$ and let $F$ be the function~\eqref{eqn:defn_F}.  If $1\in F^{-1}\left((0,0)\right)$, then the linearization $DF\colon T_1\sC^{j,\alpha}\to C^{j-2,\alpha}(X)\times C^{j-2,\alpha}(M)$ of $F$ at $w=1$ is given by
 \begin{multline*}
  DF(v) = \biggl( -\delta\left(T_{k-1}(\nabla v)\right), \\ T_{k-1}(\eta,\nabla v) - \odelta\left(S_{k-1}(\onabla v)\right) - (2k-1)H_kv - \oint T_{k-1}(\eta,\nabla v) \biggr),
 \end{multline*}
 where all geometric quantities are computed with respect to the metric $g$ in $X$ and the induced metric $h=g\rv_{TM}$ on $M$, as appropriate.
\end{prop}

\begin{proof}
 Let $v\in T_1\sC^{j,\alpha}$ and let $t\mapsto w_t$ be a smooth path in $\sC^{j,\alpha}$ such that $w_0=1$ and $\frac{\partial w_t}{\partial t}\rv_{t=0}=v$.  Denote $g_t:=g_{w_t}$.  Since $1\in F^{-1}\left((0,0)\right)$, we immediately conclude from Lemma~\ref{lem:conformal_variation} that
 \begin{align*}
  \left.\frac{\partial}{\partial t}\right|_{t=0}\sigma_k^{g_t} & = -\delta\left(T_{k-1}(\nabla v)\right), \\
  \left.\frac{\partial}{\partial t}\right|_{t=0}H_k^{g_t} & = -(2k-1)H_kv + T_{k-1}(\eta,\nabla v) - \odelta\left(S_{k-1}(\nabla v)\right) .
 \end{align*}
 Lemma~\ref{lem:conformal_variation} and integration by parts also imply that
 \[ \left.\frac{d}{dt}\right|_{t=0}\oint_M H_k^{g_t}\dvol_{h} = \oint_M \bigl((1-2k)H_k v + T_{k-1}(\eta,\nabla v)\bigr)\,\dvol_{h} . \]
 Since $H_k$ is constant and $\oint v=0$, the first summand above integrates to zero.  Combining these two displays with the definition of $F$ yields the desired conclusion.
\end{proof}

For our purposes, it is more convenient to regard $F$ as a function $\mF$ defined only on the boundary $M$.  This is done via the following definition, which makes sense as a consequence of Proposition~\ref{prop:mV_banach}.

\begin{defn}
 Let $(X^{n+1},g)$ be a compact Riemannian manifold with boundary $(M^n,h)$ and let $k\in\bN$.  Suppose additionally that $1\in\sV_k^{j,\alpha}$.  The \emph{$\mH_k$-curvature of $(M,h)$} is
 \[ \mH_k^h := H_k^{\cg_1}, \]
 where $\cg_1:=e^{2\cw}g$ is determined by the extension $\cw$ of $w=1$ as in Proposition~\ref{prop:mV_banach}.
\end{defn}

Note that the $\mH_k$-curvature depends only on the conformal class $[g]$ and the choice of boundary metric $h$.  Note also that $\mH_1^{h}$ is precisely the fractional $Q$-curvature of order $1$; see~\cite{GuillarmouGuillope2007}.

By counting derivatives, we see that for $j\geq4$, the function $\mF^h\colon\sV_k^{j,\alpha}\to\sT_1^{j-2,\alpha}$ defined by
\begin{equation}
 \label{eqn:defn_mF}
 \mF^h(w) := \mH_k^{h_w} - \oint_M \mH_k^{h_w}\,\dvol_h
\end{equation}
is well-defined, where $h_w:=e^{2w}h$. Recall the definition of 
$\sT_w^{j,\alpha}$ is given in \eqref{eqn:mT_defn}.
We emphasize that the integration is taken with respect to the background metric $h$, and recall that we assume $\Vol_h(M)=1$.  Note that $\mF^h(w)=\pi_2F^g(\cw)$ for all $w\in\sV_k^{j,\alpha}$, where $\pi_2$ is the projection onto the second factor.  We omit the superscript on $\mF$ when the background metric $h$ is clear by context.

If $w=1\in\sV_k^{j,\alpha}$, then $\mH_k^{h_1}$ is constant if and only if $1\in\mF^{-1}(0)$.  The linearization of $\mF$ at $1\in\mF^{-1}(0)$ is readily computed using Proposition~\ref{prop:interior_dF}.

\begin{cor}
 \label{cor:boundary_dF}
 Let $(M^n,h)$ be the boundary of a compact Riemannian manifold $(X^{n+1},g)$; i.e.\ $M=\partial X$ and $h=\iota^\ast g$.  Define $\mF\colon\sV_k^{j,\alpha}\to \sT_1^{j-2,\alpha}$ by~\eqref{eqn:defn_mF}.  Suppose $1\in\mF^{-1}(0)$.  Then the linearization $D\mF\colon T_1\sV_k^{j,\alpha}\to\sT_1^{j-2,\alpha}$ of $\mF$ at $w=1$ is given by
 \begin{equation}
  \label{eqn:boundary_dF}
  D\mF(\phi) = T_{k-1}(\eta,\nabla v_\phi) - \odelta\left(S_{k-1}(\onabla\phi)\right) - (2k-1)\mH_k\phi,
 \end{equation}
 where $v_\phi$ is the solution of~\eqref{eqn:linearized_extension} and all geometric quantities are computed with respect to $h$ and the canonical extension $\cg$.
\end{cor}

\begin{proof}
 Let $\phi\in T_w\sV_k^{j,\alpha}$ and let $t\mapsto w_t$ be a smooth path in $\sV_k^{j,\alpha}$ such that $w_0=1$ and $\left.\frac{\partial}{\partial t}\right|_{t=0}w_t=\phi$.  Let $\cw_t$ denote the canonical extension~\eqref{eqn:metric_extension_interior} of $w_t$ and set $v:=\left.\frac{\partial}{\partial t}\right|_{t=0}\cw_t$.  Since $\sigma_k^{\cg_{w_t}}\equiv 0$ for all $t$, differentiating at $t=0$ and applying Lemma~\ref{lem:conformal_variation} implies that $v$ solves~\eqref{eqn:linearized_extension}.  As a solution of~\eqref{eqn:linearized_extension}, we conclude that
 \[ \oint_M T_{k-1}(\eta,\nabla v)\,\dvol = 0 . \]
 Combining these observations with Proposition~\ref{prop:interior_dF} yields
 \[ DF(v) = \left(0, T_{k-1}(\eta,\nabla v) - \odelta\left(S_{k-1}(\onabla\phi)\right) - (2k-1)H_k\phi\right) . \]
 The observation that $F(\cw_t)=\left(0,\mF(w_t)\right)$ then yields~\eqref{eqn:boundary_dF}.
\end{proof}

The $\mH_k$-curvature --- or equivalently, the pair $(\sigma_k;H_k)$ --- is variational in a conformal class if and only if the linearization $D\mF_w$ is formally self-adjoint with respect to the $L^2$-pairing induced by $h_w$ for every $w\in\sV_{k}:=\cap_j\sV_k^{j,\alpha}$ (cf.\ \cite{BransonGover2008}).  Here we only need to know that $D\mF_w$ is formally self-adjoint at a critical point $w\in\mF^{-1}(0)$ in the cases when $(\sigma_k;H_k)$ is variational.  This is an easy consequence of Lemma~\ref{lem:conformal_variation} and Corollary~\ref{cor:boundary_dF}.

\begin{prop}
 \label{prop:fsae}
 Let $(M^n,h)$ be the boundary of a compact Riemannian manifold $(X^{n+1},g)$ and let $\mF$ be as in~\eqref{eqn:defn_mF}.  Suppose $1\in\mF^{-1}(0)$.  Then $D\mF$ at $w=1$ is formally self-adjoint with respect to the $L^2(\dvol_{h})$-pairing; i.e.
 \[ \oint_M \phi\,D\mF(\psi)\,\dvol_{h} = \oint_M \psi\,D\mF(\phi)\,\dvol_{h} \]
 for all $\phi,\psi\in C^\infty(M)$.
\end{prop}

\begin{proof}
 Given $\phi,\psi\in C^\infty(M)$, let $v_\phi$ and $v_\psi$ be extensions as in~\eqref{eqn:linearized_extension}.  The divergence theorem then implies that
 \[ \oint_M \phi\,T_{k-1}(\eta,\nabla v_\psi)\,\dvol_{h} = \int_X T_{k-1}(\nabla v_\phi,\nabla v_\psi)\,\dvol_{g} . \]
 It readily follows from Corollary~\ref{cor:boundary_dF} that
 \begin{multline*}
  \oint_M \phi\,D\mF(\psi)\,\dvol_{h} = \int_X T_{k-1}(\nabla v_\phi,\nabla v_\psi)\,\dvol_{g} \\ + \oint_M \left( S_{k-1}(\onabla\phi,\onabla\psi) - (2k-1)\mH_k\phi\psi\right)\,\dvol_{h} .
 \end{multline*}
 The right-hand side of the above display is clearly symmetric in $\phi,\psi$, which yields the desired conclusion.
\end{proof}

%% file: fredholm.tex
\section{The Fredholm property}
\label{sec:fredholm}

The goal of this section is to prove that the linearization of $\mF$ is a Fredholm operator.  Since the property of being a Fredholm operator is sensitive to both the domain and codomain, we need to handle separately the cases $k=1$ and $k\geq2$.  We begin by considering the case $k\geq2$.

\begin{thm}
 \label{thm:fredholm}
 Fix integers $k\geq2$ and $j\geq4$ and a parameter $\alpha\in(0,1)$.  Let $(X^{n+1},g)$ be a compact Riemannian manifold with boundary $M^n:=\partial X$.  Suppose that $1\in\sV_{k}^{j,\alpha}$ and that the extension $\cw$ of $w=1$ by~\eqref{eqn:metric_extension_interior} has $S_{k-1}^{\cg_1}>0$.  Then $D\mF\colon \sT_1^{j,\alpha}\to\sT_1^{j-2,\alpha}$ is Fredholm of index zero.
\end{thm}

The main idea of the proof of Theorem~\ref{thm:fredholm} is that the assumption that $T_{k-1}^{\cg_1}$ and $S_{k-1}^{\cg_1}$ are positive definite ensures that $D\mF$ is elliptic.  By Proposition~\ref{prop:fsae}, $D\mF$ is also formally self-adjoint.  Together these facts imply that $D\mF$ is Fredholm of index zero.  The subtlety here is that $D\mF$ is a nonlocal operator.  Since there does not seem to be a direct reference which guarantees that $D\mF$ is Fredholm, we sketch the proof.  The first step is to use ellipticity to deduce Schauder estimates for $D\mF$.

\begin{prop}
 \label{prop:schauder}
 Fix integers $k\geq2$ and $j\geq4$ and a parameter $\alpha\in(0,1)$.  Let $(X^{n+1},g)$ be a compact Riemannian manifold with boundary $M^n:=\partial X$.  Suppose that $1\in\sV_{k}^{j,\alpha}$ and that the extension $\cw$ of $w=1$ by~\eqref{eqn:metric_extension_interior} has $S_{k-1}^{\cg_1}>0$.  Then there is a uniform constant $C>0$ such that
 \begin{equation}
  \label{eqn:schauder}
  \lV \phi\rV_{C^{j,\alpha}(M)} \leq C\left( \lV D\mF(\phi)\rV_{C^{j-2,\alpha}(M)} + \lV \phi\rV_{C^{0,\alpha}(M)} \right)
 \end{equation}
 for all $\phi\in\sT_1^{j,\alpha}$.
\end{prop}

\begin{proof}
 In what follows, all geometric quantities in $X$ are determined by the smooth metric $\cg_1:= e^{2\cw}g$ and all geometric quantities on $M$ are determined by the metric $h$.  Instead of considering the operator $D\mF$ as defined in Corollary~\ref{cor:boundary_dF}, we consider the equivalent interior operator of Proposition~\ref{prop:interior_dF} by using the extension~\eqref{eqn:linearized_extension} from $\phi\in C^{j,\alpha}(M)$ to $v_\phi\in C^{j,\alpha}(X)$.  That is, $v_\phi$ solves
 \begin{equation}
  \label{eqn:case2}
  \begin{cases}
   \delta\left(T_{k-1}(\nabla v)\right) = 0, & \text{in $X$}, \\
   B_k(v) = D\mF(\phi), & \text{on $M$},
  \end{cases}
 \end{equation}
 where
 \[ B_k(v) := T_{k-1}(\eta,\nabla v) - \odelta\left(S_{k-1}(\onabla\iota^\ast v)\right) - (2k-1)H_k\iota^\ast v . \]
 Since $S_{k-1}>0$, we see that~\eqref{eqn:case2} satisfies the Complementing Condition of Agmon, Douglis and Nirenberg~\cite{AgmonDouglisNirenberg1959} (also known as the Lopatinskii--Shapiro conditions~\cite{Lopatinskii1953,Shapiro1953}).  The boundary Schauder estimate~\cite[Theorem~7.3]{AgmonDouglisNirenberg1959} of Agmon--Douglis--Nirenberg states that there is a uniform constant $C>0$ such that
 \begin{equation}
  \label{eqn:3.4}
  \lV v_\phi\rV_{C^{j,\alpha}(X)} \leq C\left( \lV D\mF(\phi)\rV_{C^{j-2,\alpha}(M)} + \lV v_\phi\rV_{C^{0,\alpha}(X)} \right) .
 \end{equation}
 Since $v_\phi$ is the extension of $\phi$ by \eqref{eqn:linearized_extension}, there is a uniform constant $C>0$ (cf.\ Lemma~\ref{lem:embedding}) such that
 \[ \lV v_\phi\rV_{C^{0,\alpha}(X)} \leq C\lV\phi\rV_{C^{0,\alpha}(M)}.\]
 Combining this with~\eqref{eqn:3.4} yields~\eqref{eqn:schauder}.
\end{proof}

The remaining steps in the proof of Theorem~\ref{thm:fredholm} are to use the Schauder estimate~\eqref{eqn:schauder} to show that $D\mF$ is a Fredholm operator and then use formal self-adjointness to conclude that the Fredholm index is zero.  We sketch these details below:

\begin{proof}[Proof of Theorem~\ref{thm:fredholm}]
 Let $B\subset C^{j,\alpha}(M)\cap\ker D\mF$ be the unit ball in $\ker D\mF$ with respect to the $C^{j,\alpha}(M)$-norm.  Since the embedding of $C^{j,\alpha}(M)$ into $C^{0,\alpha}(M)$ is compact, $B$ is precompact in $C^{0,\alpha}(M)$.  It then follows from~\eqref{eqn:schauder} that $B$ is precompact in $C^{j,\alpha}(M)$.  Therefore $\ker D\mF$ is finite dimensional.  Note, in fact, that all elements of $\ker D\mF$ are smooth.

 Let $\left(\ker D\mF\right)^\perp$ denote the orthogonal complement of $\ker D\mF$ with respect to the $L^2$-inner product.  We claim that there is a uniform constant $C>0$ such that
 \begin{equation}
  \label{eqn:poincare}
  \lV \phi\rV_{C^{0,\alpha}(M)} \leq C\lV D\mF(\phi)\rV_{C^{j-2,\alpha}(M)}
 \end{equation}
 for all $\phi\in \left(\ker D\mF\right)^\perp$.  If not, then there would be a sequence $(\phi_\ell)_\ell\subset\left(\ker D\mF\right)^\perp$ such that $\lV \phi_\ell\rV_{C^{0,\alpha}(M)}=1$ and $D\mF(\phi_\ell)\to0$ in $C^{j-2,\alpha}(M)$.  It follows
 from~\eqref{eqn:3.4} that $v_{\phi_\ell}$ is bounded in $C^{j,\alpha}(X)$. Thus, up to a subsequence, $\phi_\ell$ converges strongly in $C^{0,\alpha}(M)$ and thus in $L^2(M)$, say to $\phi$.  In particular, $\lV\phi\rV_{C^{0,\alpha}(M)}=1$, and thus 
 $\lV\phi\rV_{L^2(M)}\leq C$. Since $D\mF(\phi_\ell)\to0$, the limit $\phi$ is a weak solution of $D\mF(\phi)=0$.  It follows from elliptic estimates as in the proof of Proposition~\ref{prop:schauder} that $\phi$ is a strong solution of $D\mF(\phi)$; i.e.\ $\phi\in\ker D\mF$.  Therefore
 \[ \oint_M \phi^2 = \oint_M \left(\phi-\phi_\ell\right)\phi \leq \lV\phi-\phi_\ell\rV_{L^2(M)}\lV\phi\rV_{L^2(M)} . \]
 Hence $\phi=0$, a contradiction.

 It follows immediately from~\eqref{eqn:schauder} and~\eqref{eqn:poincare} that the image of $D\mF$ is closed.  Note now that
 \[ \left(\coker D\mF\right)^\ast = \left\{ \beta\in \left(C^{j-2,\alpha}(M)\right)^\ast \suchthat \beta\circ D\mF\equiv 0 \right\} . \]
 Let $K=\ker\left((D\mF)^\ast\colon C^\infty(M)\to C^\infty(M)\right)$ be the kernel of the formal adjoint of $D\mF$ and denote by $i\colon K\to\left(\coker D\mF\right)^\ast$ the $L^2$-embedding
 \[ i(v)(\phi) = \oint_M v\phi . \]
 As in the previous paragraph, elliptic regularity implies that if $\beta\in\left(\coker D\mF\right)^\ast$, then $\beta$ is smooth, and hence $\beta\in K$.  Thus $i$ is a bijection, and so we may identify $K\cong\left(\coker D\mF\right)^\ast$.  Since $D\mF$ is formally self-adjoint, we see that $K=\ker D\mF$.  By the first paragraph, this is finite-dimensional, and hence $D\mF$ is Fredholm.  Moreover,
 \[ \FredholmIndex D\mF = \dim\ker D\mF - \dim\coker D\mF = 0 . \qedhere \]
\end{proof}

We can also prove that the linearization of $\mF$ is Fredholm in the case $k=1$.  In this case, recall (cf.\ Remark~\ref{rk:mapping_properties}) that $\mF\colon\sV_1^{j,\alpha}\to\sT_1^{j-1,\alpha}$, that $T_0=g$, and that $S_0\equiv0$.  The proof that $D\mF$ is Fredholm in this case is similar to the proof of Theorem~\ref{thm:fredholm}, so we give only a brief sketch.  Again, the main ingredient is a Schauder estimate for $D\mF$.

\begin{prop}
 \label{prop:schauder_Sflat}
 Fix an integer $j\geq4$ and a parameter $\alpha\in(0,1)$.  Let $(X^{n+1},g)$ be a compact Riemannian manifold with boundary $M^n:=\partial X$.  Suppose that $1\in\sV_{1}^{j,\alpha}$.  Then there is a uniform constant $C>0$ such that
 \[ \lV \phi\rV_{C^{j,\alpha}(M)} \leq C\left( \lV D\mF(\phi)\rV_{C^{j-1,\alpha}(M)} + \lV \phi\rV_{C^{0,\alpha}(M)} \right) \]
 for all $\phi\in\sT_1^{j,\alpha}$.
\end{prop}

\begin{proof}
 As in the second paragraph of the proof of Proposition~\ref{prop:schauder}, we consider a solution $v_\phi$ of~\eqref{eqn:case2} with
 \[ B_1(v) := \eta v - H_1\iota^\ast v . \]
 Standard Schauder estimates yield a uniform constant such that
 \[ \lV v_\phi\rV_{C^{j,\alpha}(X)} \leq C \left( \lV D\mF(\phi)\rV_{C^{j-1,\alpha}(M)} + \lV v_\phi\rV_{C^{0,\alpha}(X)}\right) . \]
 The conclusion readily follows.
\end{proof}

Mimicking the proof of Theorem~\ref{thm:fredholm}, but using Proposition~\ref{prop:schauder_Sflat} in place of Proposition~\ref{prop:schauder}, immediately yields the proof that $D\mF$ is Fredholm when $k=1$.

\begin{thm}
 \label{thm:fredholm_Sflat}
 Fix an integer $j\geq4$ and a parameter $\alpha\in(0,1)$.  Let $(X^{n+1},g)$ be a compact Riemannian manifold with boundary $M^n:=\partial X$.  Suppose that $1\in\sV_{1}^{j,\alpha}$.  Then $D\mF\colon\sT_1^{j,\alpha}\to\sT_1^{j-1,\alpha}$ is Fredholm of index zero.
\end{thm}

%% file: bifurcation.tex
\section{The Bifurcation Theorem}
\label{sec:bifurcation}

We are now prepared to prove our main result, Theorem~\ref{thm:bifurcation}.  This theorem gives sufficient conditions to conclude that a family of solutions to~\eqref{eqn:constant_Hk_problem} has a bifurcation instant.

\begin{defn}
 \label{defn:bifurcation_instant}
 Let $X^{n+1}$ be a compact manifold with nonempty boundary $M^n:=\partial X$.  Fix integers $k\in\bN$ and $j\geq4$, and a parameter $\alpha\in(0,1)$.  Let $\{g_s\}_{s\in[a,b]}$ be a smooth one-parameter family of $C^{j,\alpha}$-metrics on $X$ such that $\sigma_k^{g_s}\equiv0$ and with respect to which $M$ has unit volume and constant $H_k$-curvature.  A \emph{bifurcation instant for the family $\{g_s\}$} is a point $s_\ast\in(a,b)$ such that there exist sequences $(s_\ell)_\ell\subset[a,b]$ and $(w_\ell)_\ell\subset\sV_{k}^{j,\alpha}$ such that
 \begin{enumerate}
  \item $\sigma_k^{g_\ell}\equiv 0$ and $H_k^{g_\ell}$ is constant, where $g_\ell:=e^{2w_\ell}g_{s_\ell}$,
  \item $w_\ell\not\equiv0$ for all $\ell\in\bN$,
  \item $s_\ell\to s_\ast$ as $\ell\to\infty$,
  \item $w_\ell\to0$ in $\sV_{k}^{j,\alpha}$ with respect to $C^{j,\alpha}$-norm, as $\ell\to\infty$.
 \end{enumerate}
\end{defn}

In particular, if $s_\ast$ is a bifurcation instant for a family $\{g_s\}$ of metrics as in Definition~\ref{defn:bifurcation_instant}, then for each $\ell\in\bN$, there are nonhomothetic metrics in each conformal class $[g_{s_\ell}]$ which lie in $\overline{\Gamma_k^+}$, have $\sigma_k\equiv0$, and have $H_k$ constant. This yields multiple nonisometric solutions of~\eqref{eqn:constant_Hk_problem} in each conformal class $[g_{s_\ell}]$; see Corollary~\ref{cor:bifurcation_genl_volume} below for further details.

Our result relies on the fiber bundle analogue, proven by de Lima, Piccione and Zedda~\cite{deLimaPiccioneZedda2012}, of the general bifurcation theorem of Smoller and Wasserman~\cite{SmollerWasserman1990}.  In order to apply their results, we need to study the index of the Jacobi operator $D\mF$ of a solution $(X^{n+1},g)$ of~\eqref{eqn:constant_Hk_problem}.  This terminology reflects the close relationship between the second variation of the functional $G_k$ given by~\eqref{eqn:mFk_noncritical_dimension} at a solution $g$ of~\eqref{eqn:constant_Hk_problem} and the linearization $D\mF$ of the functional $\mF$ at $w=1$.

\begin{defn}
 \label{defn:index_DF}
 Fix $k\in\bN$.  Let $(X^{n+1},g)$ be a compact Riemannian manifold with boundary $M^n:=\partial X$ such that $\sigma_k\equiv0$ and $H_k$ is constant; if $k\geq3$, assume additionally that $g$ is locally conformally flat.  Suppose that either
 \begin{enumerate}
  \item $T_{k-1}>0$ and $S_{k-1}>0$, or
  \item $k=1$.
 \end{enumerate}
 The \emph{index $\Ind\bigl(D\mF\bigr)$ of the Jacobi operator} is the number of negative eigenvalues of $D\mF\colon\sT_1\to\sT_1$, where $\sT_1:=\bigcap_j\sT_1^{j,\alpha}$.
\end{defn}

A crucial point is that under the assumptions on $T_{k-1}$ and $S_{k-1}$ given in Theorem~\ref{thm:fredholm}, the index of the Jacobi operator is always finite.

\begin{lem}
 \label{lem:finite_index}
 Fix $k\in\bN$.  Let $(X^{n+1},g)$ be a compact Riemannian manifold with boundary $M^n:=\partial X$ such that $\sigma_k\equiv0$ and $H_k$ is constant; if $k\geq3$, assume additionally that $g$ is locally conformally flat.  Suppose that either
 \begin{enumerate}
  \item $T_{k-1}>0$ and $S_{k-1}>0$, or
  \item $k=1$.
 \end{enumerate}
 Then $D\mF\colon\sT_1\to\sT_1$ admits an orthonormal basis of eigenfunctions with eigenvalues $\lambda_1\leq\lambda_2\leq\lambda_3\leq\dotsb$ tending to $\infty$.  In particular, $\Ind\bigl(D\mF\bigr)$ is finite.
\end{lem}

\begin{proof}
 Define $D\colon C^\infty(M)\to C^\infty(M)$ by
 \[ D(\phi) = T_{k-1}(\eta,\nabla v_\phi) - \odelta\left(S_{k-1}(\onabla\phi)\right) . \]
 where $v_\phi$ is the extension of $\phi$ by~\eqref{eqn:linearized_extension}.  From the proofs of Proposition~\ref{prop:schauder} and Proposition~\ref{prop:schauder_Sflat}, we deduce that $D$ is a formally self-adjoint elliptic operator.  Moreover, the assumptions on $T_{k-1}$ and $S_{k-1}$ imply that the eigenvalues of $D$ are nonnegative and tend to infinity, and that $D$ is diagonalizable by eigenfunctions.  The conclusion follows from the fact that $D\mF=D-(2k-1)H_k$ differs from $D$ by a constant.
\end{proof}

It is important to know for which manifolds zero is not an eigenvalue of $D\mF$.

\begin{defn}
 Fix $k\in\bN$.  Let $(X^{n+1},g)$ be a compact Riemannian manifold with boundary $M^n=\partial X$ such that $\sigma_k\equiv0$ and $H_k$ is constant; if $k\geq3$, assume additionally that $g$ is locally conformally flat.  Suppose that either
 \begin{enumerate}
  \item $T_{k-1}>0$ and $S_{k-1}>0$, or
  \item $k=1$.
 \end{enumerate}
 The Jacobi operator $D\mF$ is \emph{nondegenerate} if $0$ is not an eigenvalue of $D\mF\colon\sT_1\to\sT_1$.
\end{defn}

We are now able to prove Theorem~\ref{thm:bifurcation}, restated below for convenience.  For the purposes of applications of this result, we have given a minimal set of conditions one must check in order to conclude the existence of a bifurcation instant.

\begin{thm}
 \label{thm:bifurcation_restatement}
 Fix $k\in\bN$, $4\leq j\in\bN$, and $\alpha\in(0,1)$.  Let $X^{n+1}$ be a compact manifold with boundary $M^n:=\partial X$.  Let $\{g_s\}_{s\in[a,b]}$ be a smooth one-parameter family of $C^\infty$-metrics on $X$ such that $\sigma_k^{g_s}\equiv0$ and with respect to which $M$ has unit volume and constant $H_k$-curvature for all $s\in[a,b]$.  If $k\geq3$, assume additionally that $g_s$ is locally conformally flat for all $s\in[a,b]$.  Suppose that:
 \begin{enumerate}
  \item for every $s\in[a,b]$, the metric $g_s\rv_{TM}$ is $k$-admissible;
  \item for every $s\in[a,b]$, either
  \begin{enumerate}
   \item $T_{k-1}^{g_s}>0$ and $S_{k-1}^{g_s}>0$, or
   \item $k=1$;
  \end{enumerate}
  \item the Jacobi operators $D\mF^{g_a}$ and $D\mF^{g_b}$ are nondegenerate; and
  \item $\Ind\bigl(D\mF^{g_a}\bigr)\not=\Ind\bigl(D\mF^{g_b}\bigr)$.
 \end{enumerate}
 Then there exists a bifurcation instant $s_\ast\in(a,b)$ for the family $\{g_s\}$.
\end{thm}

\begin{proof}
 The proof is by the Fiberwise Bifurcation Theorem~\cite[Theorem~A.2]{deLimaPiccioneZedda2012}, as we now explain.  By Proposition~\ref{prop:mV_banach}, the space $\sV_{k,s}^{j,\alpha}$ defined in terms of the metric $g_s$ is a Banach space.  Denote $\sV^{j,\alpha}:=\bigcup_s \sV_{k,s}^{j,\alpha}$ and $\sT^{j,\alpha}:=\bigcup_s\sT_s^{j,\alpha}$.  Define $\Phi\colon\sV^{j,\alpha}\to\sT^{j-i,\alpha}$ by
 \[ \Phi(s,f) =(s, \mF^{g_s}(f)) , \]
 where $\mF^{g_s}$ is defined by~\eqref{eqn:defn_mF} in terms of the background metric $g_s$ and
 \[ i = \begin{cases}
         1, & \text{if $k=1$}, \\
         2, & \text{otherwise}.
        \end{cases} \]
 Then $\Phi(s,1)=(s, 0)$ for all $s\in[a,b]$.  By Lemma~\ref{lem:finite_index} and either Theorem~\ref{thm:fredholm} or Theorem~\ref{thm:fredholm_Sflat} corresponding to the cases $i=2$ or $i=1$, respectively, the linearizations $D\Phi(s,\cdot)=D\mF^{g_s}$ are Fredholm operators which are diagonalizable by eigenfunctions.  Moreover, since the Jacobi operators $D\mF^{g_a}$ and $D\mF^{g_b}$ are nondegenerate with unequal indices, all of the conditions of the Fiberwise Bifurcation Theorem~\cite[Theorem~A.2]{deLimaPiccioneZedda2012} are met.  The conclusion of the Fiberwise Bifurcation Theorem gives the existence of the bifurcation instant $s_\ast$.
\end{proof}

Since $\sigma_k$, $H_k$, and $D\mF^g$ are homogeneous with respect to homothetic rescalings of the metric, we could equally as well study bifurcation instants for smooth families of solutions of~\eqref{eqn:constant_Hk_problem}.  
It is in this form that our bifurcation theorem is most useful, as then we do not need to explicitly normalize the family.  Here the definition of a bifurcation instant is the obvious modification of Definition~\ref{defn:bifurcation_instant}.

\begin{cor}
 \label{cor:bifurcation_genl_volume}
 Fix $k\in\bN$, $4\leq j\in\bN$, and $\alpha\in(0,1)$.  Let $X^{n+1}$ be a compact manifold with boundary $M^n:=\partial X$.  Let $\{g_s\}_{s\in[a,b]}$ be a smooth one-parameter family of $C^\infty$-metrics on $X$ such that $\sigma_k^{g_s}\equiv0$ and with respect to which $M$ has constant $H_k$-curvature for all $s\in[a,b]$.  If $k\geq3$, assume additionally that $g_s$ is locally conformally flat for all $s\in[a,b]$.  Assume that:
 \begin{enumerate}
  \item for every $s\in[a,b]$, the metric $g_s\rv_{TM}$ is $k$-admissible;
  \item for every $s\in[a,b]$, either
  \begin{enumerate}
   \item $T_{k-1}^{g_s}>0$ and $S_{k-1}^{g_s}>0$, or
   \item $k=1$;
  \end{enumerate}
  \item $\ker D\mF^{g_a}, \ker D\mF^{g_b} \subset \bR$, where $\bR$ denotes the space of constant functions on $M$ and $D\mF$ is considered as an operator on $C^{j,\alpha}(M)$ by the formula in Corollary \ref{cor:boundary_dF}; and
  \item $\Ind\bigl(D\mF^{g_s}\bigr)\not=\Ind\bigl(D\mF^{g_s}\bigr)$, where the index is computed on $\bR^\perp$, the $L^2$-orthogonal complement of the constant functions.
 \end{enumerate}
 Then there exists a bifurcation instant $s_\ast\in(a,b)$ for the family $\{g_s\}$.
\end{cor}

\begin{proof}
 Note that if $g$ is a Riemannian metric on $X$ and $c\in\bR$, then the homothetically rescaled metric $g_c:=e^{2c}g$ is such that
 \begin{align*}
  \sigma_k^{g_c} & = e^{-2kc}\sigma_k^g, \\
  H_k^{g_c} & = e^{-(2k-1)c}H_k^g , \\
  D\mF^{g_c} & = e^{-(2k-1)c}D\mF^{g_c} .
 \end{align*}
 In particular, the properties $\sigma_k^g\equiv0$ and $H_k^g$ constant are both invariant with respect to homotheties, as are the isomorphism $T_1\sV_k^{j,\alpha}=\sT_1^{j,\alpha}$ and the index $\Ind\bigl(D\mF^{g}\bigr)$ of the Jacobi operator $D\mF^g$.  Since the tangent space to the space of homotheties is isometric to the constant functions, we see that we may homothetically rescale the family $\{g_s\}$ so that it satisfies the hypotheses of Theorem~\ref{thm:bifurcation_restatement}.  Moreover, the resulting bifurcation instant $s_\ast$ for the rescaled family is also a bifurcation instant for the original family.
\end{proof}

As stated, Theorem~\ref{thm:bifurcation_restatement} and Corollary~\ref{cor:bifurcation_genl_volume} require a one-parameter family of Riemannian metrics on a fixed manifold with boundary satisfying a number of hypotheses.  We will find it useful to instead construct examples by choosing a one-parameter family of domains in a fixed Riemannian manifold.  The following version of Theorem~\ref{thm:bifurcation_restatement} applies to examples constructed this way.

\begin{cor}
 \label{cor:bifurcation}
 Fix $k\in\bN$, $4\leq j\in\bN$, and $\alpha\in(0,1)$.  Let $a\in\bR_+$ and denote by $\oB^{n+1}(a)$ the closed ball of radius $a$ in $\bR^{n+1}$.  Let $(N^m,g_N)$ be a compact Einstein manifold and suppose that there is an odd smooth function $f\colon(-a,a)\to\bR$ and an even smooth function $\psi\colon(-a,a)\to\bR_+$ such that
 \begin{equation}
  \label{eqn:warped_product}
  g := dr^2 \oplus f^2(r)\,d\theta^2 \oplus \psi^2(r)\,g_N
 \end{equation}
 defines a smooth metric on $X:=\oB^{n+1}(a)\times N^m$ such that $g\in\overline{\Gamma_k^+}$ and $\sigma_k^g\equiv0$, where $r(x)=\lv x\rv$ for $x\in\oB^{n+1}(a)$; if $k\geq3$, assume additionally that $g$ is locally conformally flat.  Given $s\in(0,a)$, set
 \[ X_s := \left\{ (x,y) \in X \suchthat r(x) \leq s \right\} \]
 and let $g_s$ denote the restriction of $g$ to $X_s$.  Assume that there are $s_1,s_2\in(0,a)$ such that $s_1<s_2$ and:
 \begin{enumerate}
  \item for every $s\in[a,b]$, the metric $g_s\rv_{TM}$ is $k$-admissible;
  \item for every $s\in[a,b]$, either
  \begin{enumerate}
   \item $T_{k-1}^{g_s}>0$ and $S_{k-1}^{g_s}>0$, or
   \item $k=1$;
  \end{enumerate}
  \item $\ker D\mF^{g_{s_1}}, \ker D\mF^{g_{s_2}} \subset \bR$, where $\bR$ denotes the space of constant functions; and
  \item $\Ind\bigl(D\mF^{g_{s_1}}\bigr)\not=\Ind\bigl(D\mF^{g_{s_2}}\bigr)$ when computed on $\bR^\perp$.
 \end{enumerate}
 Then $\partial X_s$ has constant $H_k$-curvature for all $s\in(0,a)$, and there exists a bifurcation instant $s_\ast\in(s_1,s_2)$ for the family $(X_s,g_s)$.
\end{cor}

\begin{proof}
 A straightforward computation shows that the second fundamental form of $\partial X_s$ with respect to the warped product metric~\eqref{eqn:warped_product} is
 \[ A = (ff^\prime)(r)\,d\theta^2 \oplus (\psi\psi^\prime)(r)\,g_N . \]
 Meanwhile, the well-known formula (e.g.\ \cite[(25)--(28)]{Bohm1999}) for the Ricci curvature of a multiply-warped product implies that the restriction of the Schouten tensor of $P$ to $T\partial X_s$ depends only on the scalar curvature of $g_N$ and the derivatives $f^{(j)}(s)$ and $\psi^{(j)}(s)$ for $j\in\{0,1,2\}$.  Together, this implies that the $H_k$-curvature of $\partial X_s$ is constant.

 Define $\Phi_s\colon X_s\to X$ by
 \[ \Phi_s(x,y) = \left( \frac{ax}{s}, y\right) . \]
 We readily check that $\Phi_s$ is a diffeomorphism, and hence $\hg_s:=(\Phi_s^{-1})^\ast g_s$ defines a one-parameter family of metrics on $X$.  The conclusion follows by applying Corollary \ref{cor:bifurcation_genl_volume} to $(X,\hg_s)$.
\end{proof}

%% file: example.tex
\section{Examples}
\label{sec:example}

In this section we describe some examples to which Theorem~\ref{thm:bifurcation} applies, focusing on the fully nonlinear case $k\geq2$.  Together these examples prove Theorem~\ref{thm:nonuniqueness_neg}, Theorem~\ref{thm:nonuniqueness_pos}, and Theorem~\ref{thm:nonuniqueness_zero}.  From the perspective of the boundary, our examples are conformal classes of products of Einstein manifolds, and in this way can be regarded as fully nonlinear, nonlocal analogues of bifurcation results for the Yamabe Problem~\cite{deLimaPiccioneZedda2012} and the fractional $Q$-curvature of order one~\cite{Diaz2018}.  However, the requirement that the interior metric be smooth and satisfy the degenerate elliptic equation $\sigma_k\equiv0$ precludes us from constructing the interior metric by general existence results (e.g.\ \cite{CaseWang2016s}), making it more difficult to find families to which to apply Theorem~\ref{thm:bifurcation}.  The examples we consider all arise as boundaries of the product of a positive and a negative Einstein manifold.

The first example we consider is when the interior geometry is a product of a round spherical cap with constant sectional curvature $1$ and a hyperbolic manifold with constant sectional curvature $-1$.  The boundary of such a space is a Riemannian product of a round sphere with scalar curvature dependent on the size of the cap and a hyperbolic manifold with constant sectional curvature $-1$.  The crucial point here is to choose the dimensions of each factor appropriately to obtain, for a given $k\in\bN$, a metric $g\in\overline{\Gamma_k^+}$ for which $\sigma_k\equiv0$.  By varying the size of the cap, we obtain a family of metrics to which Corollary~\ref{cor:bifurcation} can be applied.  The choice of normalization of the interior factors ensures that the product is locally conformally flat, and hence this construction gives examples for all $k\in\bN$.  We discuss the cases $k=2,3$, and thereby prove Theorem~\ref{thm:nonuniqueness_neg}, in Subsection~\ref{subsec:example/lcf} below.  The details for the cases $k\geq4$ are left to the interested reader.

The second example we consider is when the interior geometry is a product of a small geodesic ball in a hyperbolic manifold with constant sectional curvature $-1$ and a round sphere.  The boundary of such a space is a Riemannian product of round spheres, the first of which has scalar curvature dependent on the size of the geodesic ball.  As in the first example, the normalization ensures that the interior manifold is locally conformally flat, while it is possible to choose the dimensions of the factors so that the product metric lies in $\overline{\Gamma_k^+}$ and has $\sigma_k\equiv0$.  By varying the size of the geodesic ball, we obtain a family of metrics to which Corollary~\ref{cor:bifurcation} can be applied.  We discuss the cases $k=2,3$, and thereby prove Theorem~\ref{thm:nonuniqueness_pos}, in Subsection~\ref{subsec:example/lcf2} below.  The details for the cases $k\geq4$ are left to the interested reader.

The third example we consider is when the interior geometry is a product of a rotationally symmetric domain in the negative Einstein warped product $\bR^2\times_f F^{m-2}$ (see~\cite[Example~9.118(d)]{Besse}) and a round sphere.  Since this negative Einstein manifold is not locally conformally flat, this example only works for $k=2$.  However, by choosing the dimensions of the factors appropriately, we can make sure that the product metric lies in $\overline{\Gamma_2^+}$ and has $\sigma_2\equiv0$.  By varying the size of the domain in the negative Einstein factor, we obtain a family of metrics to which Corollary~\ref{cor:bifurcation} can be applied, and thereby prove Theorem~\ref{thm:nonuniqueness_zero}.

A common theme in these examples is that there are special dimensions for which the product of Einstein manifolds $(M^n,g_M)$ and $(H^m,g_H)$ with $\Ric_{g_M}=(n-1)g_M$ and $\Ric_{g_H}=-(m-1)g_H$, respectively, lie in $\overline{\Gamma_k^+}$ and have $\sigma_k\equiv0$ and $T_{k-1}>0$.

\begin{lem}
 \label{lem:einstein_products}
 Let $(M^n,g_M)$ and $(H^m,g_H)$ be Einstein manifolds with $\Ric_{g_M}=(n-1)g_M$ and $\Ric_{g_H}=-(m-1)g_H$, respectively, and let $(X^{m+n},g)$ denote their Riemannian product.  Let $\ell\in\bN$.
 \begin{enumerate}
  \item If $n=\frac{(\ell+1)(\ell+2)}{2}$ and $m=\frac{\ell(\ell+1)}{2}$, then $(X,g)$ is such that
  \[ \sigma_1 = \frac{\ell+1}{2}, \qquad \sigma_2 \equiv 0 . \]
  Moreover,
  \[ T_1 = \frac{\ell}{2}g_M \oplus \frac{\ell+2}{2}g_H . \]
  \item If $n=\frac{(\ell+1)(3\ell+2)}{2}$ and $m=\frac{\ell(3\ell-1)}{2}$, then $(X,g)$ is such that
  \[
   \sigma_1 = \frac{3\ell+1}{2}, \qquad \sigma_2 = \frac{\ell(3\ell+2)}{4}, \qquad \sigma_3 \equiv 0 . \]
  Moreover,
  \[ T_2 = \frac{\ell(3\ell-1)}{4}g_M \oplus \frac{(\ell+1)(3\ell+2)}{4}g_H . \]
  \item If $n=\frac{(\ell+1)(3\ell+4)}{2}$ and $m=\frac{\ell(3\ell+1)}{2}$, then $(X,g)$ is such that
  \[ \sigma_1 = \frac{3\ell+2}{2}, \qquad \sigma_2 = \frac{(\ell+1)(3\ell+1)}{4}, \qquad \sigma_3 \equiv 0 . \]
  Moreover,
  \[ T_2 = \frac{\ell(3\ell+1)}{4}g_M \oplus \frac{(\ell+1)(3\ell+4)}{4}g_H . \]
 \end{enumerate}
\end{lem}

\begin{remark}
 One can check that, with the normalizations of the factors as given, these are the only choices of dimensions for which the product manifold $(X,g)$ has $g\in\overline{\Gamma_k^+}$, $\sigma_k\equiv0$, and $T_{k-1}>0$ for $k\in\{2,3\}$.
\end{remark}

\begin{remark}
 \label{rk:genl_normalization}
 Given Einstein manifolds $(M^n,g_M)$ and $(H^m,g_H)$ with $\Ric_{g_M}=(n-1)g_M$ and $\Ric_{g_H}=-(m-1)g_H$, of sufficiently large dimension, one can show that there is a scalar $s\in\bR_+$ such that $(M\times H,g_M\oplus s^2g_H)$ satisfies $\sigma_1^g>0$, $\sigma_2^g\equiv0$, and $T_1^g>0$.  Such examples would allow us to relax the dimensional restriction in Theorem~\ref{thm:nonuniqueness_zero} and in the case $k=2$ of Theorem~\ref{thm:nonuniqueness_neg} and Theorem~\ref{thm:nonuniqueness_pos}.
\end{remark}

\begin{proof}
 It is straightforward to check that the Schouten tensor of $(X,g)$ is
 \[ P = \frac{1}{2}\left(g_M \oplus (-g_{H})\right) . \]
 The computations of $\sigma_k$, $k\in\{1,2,3\}$, and $T_1,T_2$ readily follow.
\end{proof}

The primary distinction between the three examples is in the choice of which Einstein factor has a boundary and what conditions are placed on that boundary.  We treat these three cases separately according to the intrinsic geometry of the boundary of the product.

\subsection{Products of a sphere and a hyperbolic manifold}
\label{subsec:example/lcf}

In this subsection we apply Theorem~\ref{thm:bifurcation} to products of a spherical cap and a hyperbolic manifold with sectional curvature $1$ and $-1$, respectively.  This normalization ensures that the product is locally conformally flat, allowing us to apply Theorem~\ref{thm:bifurcation} for general $k$.  Moreover, we choose the dimensions of the factors so that Lemma~\ref{lem:einstein_products} applies.  Our first task is to study the geometry of the boundary of these products.

\begin{lem}
 \label{lem:lcf_computations}
 Given $n,m\in\bN$, denote by $(S^{n+1},d\theta^2)$ and $(H^m,g_H)$ the round $(n+1)$-sphere of constant sectional cuvature $1$ and an $m$-dimensional manifold of constant sectional curvature $-1$, respectively.  Given $\varepsilon\in(0,\pi/2)$, set
 \[ S_\varepsilon^{n+1} = \left\{ x\in S^{n+1} \suchthat ar(x)\leq\varepsilon \right\}, \]
 where $r$ is the geodesic distance from a fixed point $p\in S^{n+1}$.  Let $(X_\varepsilon^{m+n+1},g)$ denote the Riemannian product of $(S_\varepsilon^{n+1},d\theta^2)$ and $(H^m,g_H)$, and let $\iota$ denote the inclusion of $H^m$ into $\partial X_\varepsilon$.  Let $\kappa=\cot\varepsilon$ denote the mean curvature of $\partial S_\varepsilon^{n+1}$ in $S_\varepsilon^{n+1}$ and let $\ell\in\bN$.
 \begin{enumerate}
  \item Set $n=\frac{\ell(\ell+3)}{2}$ and $m=\frac{\ell(\ell+1)}{2}$.  Then $(X_{\varepsilon},g)$ is such that
  \[ \sigma_1 = \frac{\ell+1}{2}, \qquad \sigma_2 \equiv 0, \qquad T_1 > 0 . \]
  Moreover, $g\rv_{T\partial X_\varepsilon}$ is $2$-admissible and the boundary $\partial X_\varepsilon$ is such that $H_2$ is a nonnegative constant, $S_1>0$, and
  \begin{align*}
   H_2 & = \frac{2(\ell^2+2\ell-3)!}{(\ell^2+2\ell-1)!}\binom{(\ell^2+3\ell)/2}{3}\kappa^3 + O(\kappa), \\
   \iota^\ast S_1 & = \frac{1}{\ell^2+2\ell-1}\binom{(\ell^2+3\ell)/2}{1}\kappa g_H
  \end{align*}
  as $\varepsilon\to0$.  Moreover, $H_2=0$ if $\ell=1$ while $H_2>0$ if $\ell\geq2$.
  \item Set $n=\frac{\ell(3\ell+5)}{2}$ and $m=\frac{\ell(3\ell-1)}{2}$.  Then $(X_\varepsilon,g)$ is such that
  \[ \sigma_1 = \frac{3\ell+1}{2}, \qquad \sigma_2 = \frac{\ell(3\ell+2)}{4}, \qquad \sigma_3 \equiv 0, \qquad T_2 > 0 . \]
  Moreover, $g\rv_{T\partial X_\varepsilon}$ is $3$-admissible and the boundary $\partial X_\varepsilon$ is such that $H_3$ is a nonnegative constant, $S_2>0$, and
  \begin{align*}
   H_3 & = \frac{8(3\ell^2+2\ell-5)!}{(3\ell^2+2\ell-2)!}\binom{(3\ell^2+5\ell)/2}{5}\kappa^5 + O(\kappa^3), \\
   \iota^\ast S_2 & = \left(\frac{2(3\ell^2+2\ell-4)!}{(3\ell^2+2\ell-2)!}\binom{(3\ell^2+5\ell)/2}{3}\kappa^3 + O(\kappa)\right)g_H
  \end{align*}
  as $\varepsilon\to0$.  Moreover, $H_3=0$ if $\ell=1$ while $H_3>0$ if $\ell\geq2$.
  \item Set $n=\frac{(\ell+2)(3\ell+1)}{2}$ and $m=\frac{\ell(3\ell+1)}{2}$.  Then $(X_\varepsilon,g)$ is such that
  \[ \sigma_1 = \frac{3\ell+2}{2}, \qquad \sigma_2 = \frac{(\ell+1)(3\ell+1)}{4}, \qquad \sigma_3 \equiv 0, \qquad T_2 > 0 . \]
  Moreover, $g\rv_{T\partial X_\varepsilon}$ is $3$-admissible and the boundary $\partial X_\varepsilon$ is such that $H_3$ is a positive constant, $S_2>0$, and
  \begin{align*}
   H_3 & = \frac{8(3\ell^2+4\ell-4)!}{(3\ell^2+4\ell-1)!}\binom{(3\ell^2+7\ell+2)/2}{5}\kappa^5 + O(\kappa^3), \\
   \iota^\ast S_2 & = \left(\frac{2(3\ell^2+4\ell-3)!}{(3\ell^2+4\ell-1)!}\binom{(3\ell^2+7\ell+2)/2}{3}\kappa^3 + O(\kappa)\right)g_H
  \end{align*}
  as $\varepsilon\to0$.
 \end{enumerate}
\end{lem}

\begin{proof}
 The claims about the $\sigma_k$-curvatures and the Newton tensors follow from Lemma~\ref{lem:einstein_products}.

 In terms of the coordinates $(r,\vartheta)\in(0,\varepsilon)\times S^n$ on $S_\varepsilon^{n+1}\setminus r^{-1}(0)$, we may write the metric $g$ on $X_\varepsilon$ as
 \[ g = dr^2 \oplus \sin^2r\,d\vartheta^2 \oplus g_H . \]
 Fix $s\in\bR_+$ and define $u\colon S_\varepsilon^{n+1}\times H\to\bR$ by
 \[ u(p,q) = \frac{1+sr^2(p)}{1+s\varepsilon^2} . \]
 Set $g_u:=u^{-2}g$.  A straightforward computation shows that
 \[P^{g_u} = \frac{1+4s}{2}dr^2 \oplus \frac{\sin^2 r+4sr\sin r\cos r}{2}d\vartheta^2 \oplus \left(-\frac{1}{2}\right)g_H + O(s^2) \]
 for $s$ close to zero.  
 Thus
 \[{g_u}^{-1} P^{g_u} =u^2\left[ \frac{1+4s}{2} Id_{T\bR^+} + \frac{1+4sr\cot r}{2}Id_{TS^n} + \left(-\frac{1}{2}\right)Id_{TH^m} + O(s^2)\right]\]
 Therefore
 \[u^{-2k}\sigma_k^{g_u} = \sigma_k^g + 2^{2-k}s\sum_{j=0}^{k-1}(-1)^j\binom{n}{k-1-j}\binom{m}{j}\left(1+nr\cot r\right) + O(s^2) . \]
 It follows that $g_u\in\Gamma_2^+$ for $s$ sufficiently close to zero in Case (1), and that $g_u\in\Gamma_3^+$ for $s$ sufficiently close to zero in Cases (2) and (3).  Thus $g\rv_{T\partial X_\varepsilon}$ is $2$- or $3$-admissible, as appropriate.

 Note that the second fundamental form of $\partial X_\varepsilon$ is $A=\kappa\overline{d\theta}^2$ for $\overline{d\theta}^2=d\vartheta^2\rv_{T\partial S_\varepsilon^{n+1}}$.  For any $j\in\bN$, we compute that
 \begin{align*}
  \sigma_{j,0} & = \binom{n}{j}\kappa^j, \\
  T_{j,0} & = \binom{n-1}{j}\kappa^j\overline{d\theta}^2 \oplus \binom{n}{j}\kappa^jg_{H}
 \end{align*}
 for any $j\in\bN$. We also compute that
 \begin{align*}
  T_{2,1} & = \frac{(n-1)(n-m-2)}{4}\kappa\overline{d\theta}^2 \oplus \frac{n(n-m)}{4}\kappa g_H, \\
  \sigma_{j,1} & = \frac{n-m+1-j}{2j}\binom{n}{j-1}\kappa^{j-1}, \\
  \sigma_{3,2} & = \frac{n(n^2-2mn+m^2-3n+m+2)}{24}\kappa .
 \end{align*}
 By definition,
 \begin{align*}
  H_2 & = \frac{2}{(m+n-1)(m+n-2)}\sigma_{3,0} + \frac{2}{m+n-1}\sigma_{2,1} , \\
  H_3 & = \frac{8}{(m+n-2)(m+n-3)(m+n-4)}\sigma_{5,0} + \frac{8}{(m+n-2)(m+n-3)}\sigma_{4,1} \\
   & \qquad + \frac{3}{m+n-2}\sigma_{3,2},
 \end{align*}
 and
 \begin{align*}
  S_1 & = \frac{1}{m+n-1}T_{1,0}, \\
  S_2 & = \frac{2}{(m+n-2)(m+n-3)}T_{3,0} + \frac{2}{m+n-2}T_{2,1} .
 \end{align*}
 Combining these formulae yield the claimed conclusions for $H_2,H_3$ and $S_1,S_2$.
\end{proof}

By applying Theorem~\ref{thm:bifurcation} to the examples of Lemma~\ref{lem:lcf_computations}, we obtain examples with boundary geometry the Riemannian product of a sphere and a hyperbolic manifold of varying sizes which contains infinitely many bifurcation points for~\eqref{eqn:constant_Hk_problem}.  Restricting these examples to the boundary proves Theorem~\ref{thm:nonuniqueness_neg}.

\begin{thm}
 \label{thm:lcf}
 Let $(S_\varepsilon^{n+1},d\theta^2)$, $\varepsilon\in(0,\pi/2)$ be a spherical cap, let $(H^m,g_H)$ be a compact hyperbolic manifold, and denote by $(X_\varepsilon,g)$ their Riemannian product.  Let $k\in\{2,3\}$ and suppose additionally that
 \begin{enumerate}
  \item $k=2$, $n=\frac{\ell(\ell+3)}{2}$, $m=\frac{\ell(\ell+1)}{2}$, and $\ell\geq2$;
  \item $k=3$, $n=\frac{\ell(3\ell+5)}{2}$, $m=\frac{\ell(3\ell-1)}{2}$, and $\ell\geq2$; or
  \item $k=3$, $n=\frac{(\ell+2)(3\ell+1)}{2}$, and $m=\frac{\ell(3\ell+1)}{2}$.
 \end{enumerate}
 Then, $(X_\varepsilon,g)$ is a solution of~\eqref{eqn:constant_Hk_problem} for all $\varepsilon\in(0,\pi/2)$.  Moreover, there is a sequence $(\varepsilon_j)_j\subset(0,\pi/2)$ of bifurcation points for~\eqref{eqn:constant_Hk_problem} for which $\varepsilon_j\to0$ as $j\to\infty$.
\end{thm}

\begin{proof}
 Applying Lemma~\ref{lem:lcf_computations} to $(X_\varepsilon,g)$ implies that, $(X_\varepsilon,g)$ is a solution of~\eqref{eqn:constant_Hk_problem} for all $\varepsilon\in(0,\pi/2)$.  Lemma~\ref{lem:lcf_computations} further implies that there are constants $c_1,c_2>0$ such that $\iota_2^\ast S_{k-1}=c_1\varepsilon^{3-2k}g_H+O(\varepsilon^{5-2k})$ and $H_k=c_2\varepsilon^{1-2k}+O(\varepsilon^{3-2k})$ as $\varepsilon\to0$, where $\iota\colon H^m\to\partial X_\varepsilon$ is the inclusion map.  Let $\pi\colon \partial X_\varepsilon\to H^m$ denote the projection map.  It is readily seen that for all $\phi\in C^\infty(H^m)$, the extension $v_\phi$ of $\pi^\ast\phi$ to $X_\varepsilon$ by~\eqref{eqn:linearized_extension} is of the form $v_\phi(p,q)=f\left(r(p)\right)\phi(q)$.  It readily follows that $T_{k-1}(\eta,\nabla v_\phi)=O(1)$ as $\varepsilon\to0$.  Thus
 \[ D\mF^g(\pi^\ast\phi) = \pi^\ast\left[ -\delta_{g_H}\left((\iota^\ast S_{k-1})(\onabla\phi)\right) - (2k-1)(\iota^\ast H_k)\phi\right] + O(1) . \]
 for all $\phi\in C^\infty(H^m)$.  The asymptotic behavior of $\iota^\ast S_{k-1}$ and $H_k$ imply that the index of the restriction of $D\mF$ to $H^m$, and hence the index of $D\mF$ itself, tends to $\infty$ as $\varepsilon\to0$.  Corollary~\ref{cor:bifurcation} then yields, up to scaling, the existence of the sequence $(\varepsilon_j)_j$ of bifurcation points.
\end{proof}

\subsection{Products of two spheres}
\label{subsec:example/lcf2}

In this subsection we apply Theorem~\ref{thm:bifurcation} to products of a round sphere and a small geodesic ball in hyperbolic space with sectional curvature $1$ and $-1$, respectively.  This normalization ensures that the product is locally conformally flat, allowing us to apply Theorem~\ref{thm:bifurcation} for general $k$.  Moreover, we choose the dimensions of the factors so that Lemma~\ref{lem:einstein_products} applies.  Our first task is to study the geometry of the boundary of these products.

\begin{lem}
 \label{lem:lcf_computations2}
 Given $n,m\in\bN$, denote by $(S^n,d\theta^2)$ and $(H^{m+1},g_H)$ the round $n$-sphere of constant sectional curvature $1$ the $(m+1)$-dimensional simply connected manifold of constant sectional curvature $-1$, respectively.  Given $\varepsilon\in(0,\pi/2)$, set
 \[ H_\varepsilon^{m+1} = \left\{ x\in H^{m+1} \suchthat r(x)\leq\varepsilon \right\}, \]
 where $r$ is the geodesic distance from a fixed point $p\in H^{m+1}$.  Let $(X_\varepsilon^{m+n+1},g)$ denote the Riemannian product of $(S^n,d\theta^2)$ and $(H_\varepsilon^{m+1},g_H)$, and let $\iota$ denote the inclusion of $S^n$ into $\partial X_\varepsilon$.  Let $\kappa=\coth\varepsilon$ denote the mean curvature of $\partial H_\varepsilon^{m+1}$ in $H_\varepsilon^{m+1}$ and let $\ell\in\bN$.
 \begin{enumerate}
  \item Set $n=\frac{(\ell+1)(\ell+2)}{2}$ and $m=\frac{(\ell-1)(\ell+2)}{2}$ for $\ell\geq 2$.  Then $(X_{\varepsilon},g)$ is such that
  \[ \sigma_1 = \frac{\ell+1}{2}, \qquad \sigma_2 \equiv 0, \qquad T_1 > 0 . \]
  Moreover, $g\rv_{T\partial X_\varepsilon}$ is $2$-admissible and the boundary $\partial X_\varepsilon$ is such that $H_2$ is a positive constant, $S_1>0$, and
  \begin{align*}
   H_2 & = \frac{2(\ell^2+2\ell-3)!}{(\ell^2+2\ell-1)!}\binom{(\ell-1)(\ell+2)/2}{3}\kappa^3 + O(\kappa), \ \ \mbox{if}\ \ell\geq 3, \\
   \iota^\ast S_1 & = \frac{1}{\ell^2+2\ell-1}\binom{(\ell-1)(\ell+2)/2}{1}\kappa\,d\theta^2
  \end{align*}
  as $\varepsilon\to0$.  
  \item Set $n=\frac{(\ell+1)(3\ell+2)}{2}$ and $m=\frac{(\ell-1)(3\ell+2)}{2}$ for $\ell\geq 2$.  Then $(X_\varepsilon,g)$ is such that
  \[ \sigma_1 = \frac{3\ell+1}{2}, \qquad \sigma_2 = \frac{\ell(3\ell+2)}{4}, \qquad \sigma_3 \equiv 0, \qquad T_2 > 0 . \]
  Moreover, $g\rv_{T\partial X_\varepsilon}$ is $3$-admissible and the boundary $\partial X_\varepsilon$ is such that $H_3$ is a positive constant, $S_3>0$, and
  \begin{align*}
   H_3 & = \frac{8(3\ell^2+2\ell-5)!}{(3\ell^2+2\ell-2)!}\binom{(\ell-1)(3\ell+2)/2}{5}\kappa^5 + O(\kappa^3), \ \ \mbox{if}\ \ell\geq 3,\\
   \iota^\ast S_2 & = \left(\frac{2(3\ell^2+2\ell-4)!}{(3\ell^2+2\ell-2)!}\binom{(\ell-1)(3\ell+2)/2}{3}\kappa^3 + O(\kappa)\right)\,d\theta^2
  \end{align*}
  as $\varepsilon\to0$. 
  \item Set $n=\frac{(\ell+1)(3\ell+4)}{2}$ and $m=\frac{(\ell+1)(3\ell-2)}{2}$.  Then $(X_\varepsilon,g)$ is such that
  \[ \sigma_1 = \frac{3\ell+2}{2}, \qquad \sigma_2 = \frac{(\ell+1)(3\ell+1)}{4}, \qquad \sigma_3 \equiv 0, \qquad T_2 > 0 . \]
  Moreover, $g\rv_{T\partial X_\varepsilon}$ is $3$-admissible and the boundary $\partial X_\varepsilon$ is such that $H_3$ is a positive constant, $S_3>0$, and
  \begin{align*}
   H_3 & = \frac{8(3\ell^2+4\ell-4)!}{(3\ell^2+4\ell-1)!}\binom{(\ell+1)(3\ell-2)/2}{5}\kappa^5 + O(\kappa^3), \ \ \mbox{if}\ \ell\geq 2, \\
   \iota^\ast S_2 & = \left(\frac{2(3\ell^2+4\ell-3)!}{(3\ell^2+4\ell-1)!}\binom{(\ell+1)(3\ell-2)/2}{3}\kappa^3 + O(\kappa)\right)\,d\theta^2
  \end{align*}
  as $\varepsilon\to0$.
 \end{enumerate}
\end{lem}

\begin{proof}
 The claims about the $\sigma_k$-curvatures and the Newton tensors follow from Lemma~\ref{lem:einstein_products}.

 In terms of the coordinates $(r,\vartheta)\in(0,\varepsilon)\times S^m$ on $H_\varepsilon^{m+1}\setminus r^{-1}(0)$, we may write the metric $g$ on $X_\varepsilon$ as
 \[ g = d\theta^2 \oplus dr^2 \oplus \sinh^2r\,d\vartheta^2 . \]
 Fix $s\in\bR_+$ and define $u\colon S^n \times H_\varepsilon^{m+1}\to\bR$ by
 \[ u(p,q) = \frac{1+sr^2(q)}{1+s\varepsilon^2} . \]
 Set $g_u:=u^{-2}g$.  A straightforward computation shows that
 \[ P^{g_u} = \frac{1}{2}d\theta^2 \oplus \frac{4s-1}{2}dr^2 \oplus \frac{4sr\cosh r\sinh r-\sinh^2 r}{2}d\vartheta^2 + O(s^2) \]
 for $s$ close to zero.  
 Thus 
 \[ g_u^{-1}P^{g_u} =u^{2}\left [\frac{1}{2}Id_{TS^n} + \frac{4s-1}{2} Id_{T\mathbb R^+} + \frac{4sr\coth r-1}{2}Id_{TH^m} + O(s^2)\right ]\]
 Therefore
 \[ u^{-2k}\sigma_k^{g_u} = \sigma_k^g + 2^{2-k}s\sum_{j=0}^{k-1}(-1)^{k-1-j}\binom{m}{k-1-j}\binom{n}{j}\left(1+mr\coth r\right) + O(s^2) . \]
 It follows that $g_u\in\Gamma_2^+$ for $s$ sufficiently close to zero in Case (1), and that $g_u\in\Gamma_3^+$ for $s$ sufficiently close to zero in Cases (2) and (3).  Thus $g\rv_{T\partial X_\varepsilon}$ is $2$- or $3$-admissible, as appropriate.

 Note that the second fundamental form of $\partial X_\varepsilon$ is $A=\kappa\overline{g_H}$ for $\overline{g_H}=g_H\rv_{T\partial H_\varepsilon^{m+1}}$.  For any $j\in\bN$, we compute that
 \begin{align*}
  \sigma_{j,0} & = \binom{m}{j}\kappa^j, \\
  T_{j,0} & = \binom{m}{j}\kappa^jd\theta^2 \oplus \binom{m-1}{j}\kappa^j\overline{g_H}
 \end{align*}
 for any $j\in\bN$.  We also compute that
 \begin{align*}
  T_{2,1} & = \frac{m(n-m)}{4}\kappa d\theta^2 \oplus \frac{(m-1)(n-m+2)}{4}\kappa\overline{g_H}, \\
  \sigma_{j,1} & = \frac{n-m-1+j}{2j}\binom{m}{j-1}\kappa^{j-1}, \\
  \sigma_{3,2} & = \frac{m(n^2-2mn+m^2+n-3m+2)}{24}\kappa .
 \end{align*}
 By definition,
 \begin{align*}
  H_2 & = \frac{2}{(m+n-1)(m+n-2)}\sigma_{3,0} + \frac{2}{m+n-1}\sigma_{2,1} , \\
  H_3 & = \frac{8}{(m+n-2)(m+n-3)(m+n-4)}\sigma_{5,0} + \frac{8}{(m+n-2)(m+n-3)}\sigma_{4,1} \\
   & \qquad + \frac{3}{m+n-2}\sigma_{3,2},
 \end{align*}
 and
 \begin{align*}
  S_1 & = \frac{1}{m+n-1}T_{1,0}, \\
  S_2 & = \frac{2}{(m+n-2)(m+n-3)}T_{3,0} + \frac{2}{m+n-2}T_{2,1} .
 \end{align*}
 Combining these formulae yield the claimed conclusions for $H_2,H_3$ and $S_1,S_2$.
\end{proof}

By applying Theorem~\ref{thm:bifurcation} to the examples of Lemma~\ref{lem:lcf_computations2}, we obtain examples with boundary geometry the Riemannian product of two spheres of varying sizes which contains infinitely many bifurcation points for~\eqref{eqn:constant_Hk_problem}.  Restricting these examples to the boundary proves Theorem~\ref{thm:nonuniqueness_pos}.

\begin{thm}
 \label{thm:lcf2}
 Let $(S^n,d\theta^2)$ be a round sphere with constant sectional curvature $1$ and let $(H_\varepsilon^{m+1},g_H)$, $\varepsilon\in\bR_+$, be a geodesic ball in hyperbolic space of constant sectional curvature $-1$.  Denote by $(X_\varepsilon,g)$ their Riemannian product.  Let $k\in\{2,3\}$ and suppose additionally that
 \begin{enumerate}
  \item $k=2$, $n=\frac{(\ell+1)(\ell+2)}{2}$, $m=\frac{(\ell-1)(\ell+2)}{2}$, and $\ell\geq3$;
  \item $k=3$, $n=\frac{(\ell+1)(3\ell+2)}{2}$, $m=\frac{(\ell-1)(3\ell+2)}{2}$, and $\ell\geq3$; or
  \item $k=3$, $n=\frac{(\ell+1)(3\ell+4)}{2}$, $m=\frac{(\ell+1)(3\ell-2)}{2}$ and $\ell\geq 2$.
 \end{enumerate}
 Then, $(X_\varepsilon,g)$ is a solution of~\eqref{eqn:constant_Hk_problem} for all $\varepsilon\in\bR_+$.  Moreover, there is a sequence $(\varepsilon_j)_j\subset\bR_+$ of bifurcation points for~\eqref{eqn:constant_Hk_problem} for which $\varepsilon_j\to0$ as $j\to\infty$.
\end{thm}

\begin{proof}
 Applying Lemma~\ref{lem:lcf_computations2} to $(X_\varepsilon,g)$ implies that, $(X_\varepsilon,g)$ is a solution of~\eqref{eqn:constant_Hk_problem} for all $\varepsilon\in(0,\pi/2)$.  Lemma~\ref{lem:lcf_computations2} further implies that there are constants $c_1,c_2>0$ such that $\iota^\ast S_{k-1}=c_1\varepsilon^{3-2k}d\theta^2+O(\varepsilon^{5-2k})$ and $H_k=c_2\varepsilon^{1-2k}+O(\varepsilon^{3-2k})$ as $\varepsilon\to0$, where $\iota\colon S^n\to \partial X_\varepsilon$ is the inclusion map.  Let $\pi\colon \partial X_\varepsilon\to S^n$ denote the projection map.  It is readily seen that for all $\phi\in C^\infty(S^n)$, the extension $v_\phi$ of $\pi^\ast\phi$ to $X_\varepsilon$ by~\eqref{eqn:linearized_extension} is of the form $v_\phi(p,q)=f\left(r(q)\right)\phi(p)$.  It readily follows that $T_{k-1}(\eta,\nabla v_\phi)=O(1)$ as $\varepsilon\to0$.  Thus
 \[ D\mF^g(\pi^\ast\phi) = \pi^\ast\left[ -\delta_{d\theta^2}\left((\iota^\ast S_{k-1})(\onabla\phi)\right) - (2k-1)(\iota^\ast H_k)\phi\right] + O(1) . \]
 for all $\phi\in C^\infty(S^n)$.  The asymptotic behavior of $\iota^\ast S_{k-1}$ and $H_k$ imply that the index of the restriction of $D\mF$ to $S^n$, and hence the index of $D\mF$ itself, tends to $\infty$ as $\varepsilon\to0$.  Corollary~\ref{cor:bifurcation} then yields, the existence of the sequence $(\varepsilon_j)_j$ of bifurcation points.
\end{proof}

\subsection{Products of a sphere and a Ricci flat manifold}
\label{subsec:example/flat}

In this subsection we apply Theorem~\ref{thm:bifurcation} to products of a small domain in Einstein warped products of the form $\bR^2\times_f F^{m-1}$ and an $n$-sphere.  For convenience, we normalize the Einstein constants of each factors so that we may apply Lemma~\ref{lem:einstein_products}.  Since $\mathbb R^2\times_f F^{m-1}$ is not locally conformally flat, the product is not either, and so we restrict to the case $k=2$.  In particular, one could dispense with the dimensional constraints of Lemma~\ref{lem:einstein_products} (cf.\ Remark~\ref{rk:genl_normalization}).

We first study the geometry of the boundary of these products.  To that end, given a compact manifold $(X^{n+1},g)$ with boundary $M:=\partial X$ for which $T_1>0$, we denote by $\mD\colon C^{j,\alpha}(M)\to C^{j-1,\alpha}(M)$ the Dirichlet-to-Neumann operator
\begin{equation}
 \label{eqn:dirichlet_to_neumann}
 \mD(\phi) := T_1(\eta,\nabla v_\phi),
\end{equation}
where $\eta$ is the outward-pointing unit normal along $M$ and $v_\phi$ is the extension of $\phi$ by~\eqref{eqn:linearized_extension}.

\begin{lem}
 \label{lem:flat_computations}
 Fix $2\leq\ell\in\bN$ and set $n=\frac{(\ell+1)(\ell+2)}{2}$ and $m=\frac{(\ell-1)(\ell+2)}{2}$.  Denote by $(S^n,d\theta^2)$ the round $n$-sphere of constant sectional curvature and let $(F^{m-1},g_F)$ be a Ricci flat manifold.  Given $R\in\bR_+$, denote $N_R^{m+1}:=B_R(0)\times F^{m-1}$, where $B_R(0):=\left\{x\in\bR^2\suchthat \lv x\rv^2\leq R\right\}$, and by
 \begin{equation}
  \label{eqn:flat_einstein_metric}
  g_N := dr^2 \oplus \left(\frac{2}{m}\cosh^{\frac{2-m}{m}}\bigl(\frac{mr}{2}\bigr)\sinh\bigl(\frac{mr}{2}\bigr)\right)^2\,d\vartheta^2 \oplus \cosh^{\frac{4}{m}}\bigl(\frac{mr}{2}\bigr)g_F,
 \end{equation}
 where $r$ is the distance from $0\in B_R(0)$ and $d\vartheta^2$ is the standard metric on $S^1$.  Denote by $(X_R^{m+n+1},g)$ the Riemannian product of $(S^n,d\theta^2)$ and $(N_R,g_N)$.  Denote by $\iota\colon S^1\to\partial X_R$ and $\pi\colon\partial X_R\to S^1$ the inclusion and projection, respectively, of $S^1$, as parameterized by $\vartheta$, to $\partial X_R$.  Then
 \[ \sigma_1 = \frac{\ell+1}{2}, \qquad \sigma_2 \equiv 0, \qquad T_1 > 0 . \]
 Moreover, $g\rv_{T\partial X_R}$ is $2$-admissible and the boundary $\partial X_R$ is such that $H_2$ is a positive constant, $S_1>0$, and, as $R\to\infty$, it holds that
 \begin{align*}
  \mD(\pi^\ast\phi) & = o(1), \\
  H_2 & = \frac{2}{(m+n-1)(m+n-2)}\binom{m}{3} + \frac{m(n-m+1)}{m+n-1} + o(1), \\
  \iota^\ast S_1 & = \frac{m-1}{m+n-1}\tanh\bigl(\frac{mR}{2}\bigr)\,d\vartheta^2
 \end{align*}
 for all $\phi\in C^\infty(S^1)$.
\end{lem}

\begin{proof}
 It is straightforward to check that $(N_R,g_N)$ is Einstein with $\Ric_{g_N}=-mg_N$ (cf.\ \cite[Example~9.118(d)]{Besse}).  The claims about the $\sigma_k$-curvatures and the first Newton tensor follow from Lemma~\ref{lem:einstein_products}.

 We now show that $g\rv_{T\partial X_R}$ is $2$-admissible.  It suffices to show that there is a function $u=u(r)$ such that $u^{-2}g\in\Gamma_2^+$, as then the restriction of the function $U=u(r)/u(R)$ to $X_R$ is such that $U^{-2}g\in\Gamma_2^+$ and $U\rv_{\partial X_R}=1$.  Given $c\in\bR_+$, define $u:=\cosh(cr)$ and set $\hg:=u^{-2}g$.   Denote
 \[ B := g^{-1}\left(u^{-1}\nabla^2u - \frac{1}{2}u^{-2}\lv\nabla u\rv^2 g\right), \]
 and recall that $g^{-1}P^{\hg} = g^{-1}P^g + B$.  A straightforward computation shows that
 \begin{align*}
  B & = c^2\Id_{T\bR_+} + c\tanh(cr)\left(\coth\bigl(\frac{mr}{2}\bigr)+(m-2)\csch(mr)\right)\Id_{TS^1} \\
   & \quad + c\tanh(cr)\tanh\bigl(\frac{mr}{2}\bigr)\Id_{TF} - \frac{c^2}{2}\tanh^2(cr)\Id_{TX} .
 \end{align*}
 We compute that
 \begin{align*}
  \sigma_1(B) & = c^2 + mc\tanh(cr)\coth(mr) - \frac{m+n+1}{2}c^2\tanh^2(cr), \\
  \sigma_2(B) & = mc^3\tanh(cr)\coth(mr) + \frac{m(m-1)}{2}c^2\tanh^2(cr) - \frac{m+n}{2}c^4\tanh^2(cr) \\
   & \quad - \frac{m(m+n)}{2}c^3\tanh^3(cr)\coth(mr) + \frac{(m+n)(m+n+1)}{8}c^4\tanh^4(cr) .
 \end{align*}
 It follows that there is a $c>0$ sufficiently close to zero such that $\sigma_1(B),\sigma_2(B)>0$.  It follows from the convexity of $\overline{\Gamma_2^+}$ that for this choice of $c$, it holds that $g_u\in\Gamma_2^+$, as desired.

 Next we consider the restriction of the Dirichlet-to-Neumann operator~\eqref{eqn:dirichlet_to_neumann} to functions $\phi=\phi(\vartheta)$ which depend only on the $S^1$-factor in $N$.  From Lemma~\ref{lem:einstein_products}, we see that $v_\phi$ is the solution to
 \[ \begin{cases}
     \Delta_{g_N}v = 0, & \quad \text{in $X_R$}, \\
     v = \phi, & \text{on $\partial X_R$} .
    \end{cases} \]
 By separation of variables, it holds that $v=u(r)\phi(\vartheta)$ where $u(r)$ is the unique smooth solution of the ODE
 \begin{equation}
  \label{eqn:flat_ode}
  u^{\prime\prime} + m\coth(mr)u^\prime + \left(\frac{m}{2}\sech^{\frac{2-m}{m}}\bigl(\frac{mr}{2}\bigr)\csch\bigl(\frac{mr}{2}\bigr)\right)u = 0
 \end{equation}
 on $[0,R]$ with $u(R)=1$.  In fact, we note that there is a unique solution $U$ to~\eqref{eqn:flat_ode} on $[0,\infty)$ with $U(0)=1$, and then $u(r)=U(r)/U(R)$.  Applying Lemma~\ref{lem:einstein_products} again yields
 \[ \mD(\phi) = \frac{\ell+2}{2}u^\prime(R)\phi = \frac{(\ell+2)\phi}{2}\frac{U^\prime(R)}{U(R)} . \]
 On the other hand, it follows from~\eqref{eqn:flat_ode} that
 \[ \frac{d}{dr}\left(\frac{U^\prime}{U}\right) \leq -m\coth(mr)\frac{U^\prime}{U} - \left(\frac{m}{2}\sech^{\frac{2-m}{m}}\bigl(\frac{mr}{2}\bigr)\csch\bigl(\frac{mr}{2}\bigr)\right) , \]
 from which it readily follows that $U^\prime/U\to0$ as $r\to\infty$.  In particular, $\mD(\phi)=o(1)$ as $R\to\infty$.

 Set $h:=g\rv_{T\partial X_R}$ and note that the second fundamental form $A$ of $\partial X_R$ is such that
 \[ h^{-1}A = \left(\coth\bigl(\frac{mR}{2}\bigr)+(m-2)\csch(mR)\right)\Id_{TS^1} \oplus \tanh\bigl(\frac{mR}{2}\bigr)\Id_{TF} . \]
 It follows that
 \begin{align*}
  \sigma_{3,0} & = \binom{m-1}{3}\tanh^3\bigl(\frac{mR}{2}\bigr) + \binom{m-1}{2}\left(1+\frac{m-2}{2}\sech^2\bigl(\frac{mR}{2}\bigr)\right)\tanh\bigl(\frac{mR}{2}\bigr), \\
  \sigma_{2,1} & = \frac{m(n-m+1)}{2}\coth(mR), \\
  \iota^\ast T_{1,0} & = (m-1)\tanh\bigl(\frac{mR}{2}\bigr)\,d\vartheta^2 .
 \end{align*}
 Combining this with the definitions of $H_2$ and $S_1$ yield the final conclusions.
\end{proof}

By applying Theorem~\ref{thm:bifurcation} to the examples of Lemma~\ref{lem:flat_computations}, we obtain examples with boundary geometry the Riemannian product of a sphere and a Ricci flat manifold of varying sizes which contains infinitely many bifurcation points for~\eqref{eqn:constant_Hk_problem}.  Restricting these examples to the boundary proves Theorem~\ref{thm:nonuniqueness_zero}.

\begin{thm}
 \label{thm:flat}
 Fix $2\leq\ell\in\bN$ and set $n=\frac{(\ell+1)(\ell+2)}{2}$ and $m=\frac{(\ell-1)(\ell+2)}{2}$.  Denote by $(S^n,d\theta^2)$ the round $n$-sphere of constant sectional curvature $1$ and let $(F^m,g_F)$ be a Ricci flat manifold.  Given $R\in\bR_+$, denote $N_R^{m+1}:=B_R(0)\times F^{m-1}$, where $B_R(0)$ is the closed Euclidean ball of radius $R$ in $\bR^2$, and let $g_N$ be the metric~\eqref{eqn:flat_einstein_metric} on $N_R$.  Denote by $(X_R,g)$ the Riemannian product of $(S^n,d\theta^2)$ and $(N_R,g_N)$.  Then, $(X_R,g)$ is a solution of~\eqref{eqn:constant_Hk_problem} for all $R\in\bR_+$.  Moreover, there is a sequence $(R_j)_j\subset\bR_+$ of bifurcation points for~\eqref{eqn:constant_Hk_problem} for which $R_j\to\infty$ as $j\to\infty$.
\end{thm}

\begin{proof}
 Applying Lemma~\ref{lem:flat_computations} to $(X_R,g)$ implies that, $(X_R,g)$ is a solution of~\eqref{eqn:constant_Hk_problem} for all $R\in\bR_+$.  Combining~\eqref{eqn:flat_einstein_metric} and Lemma~\ref{lem:flat_computations} yields
 \[ D\mF(\pi^\ast\phi)\rv_{S^1} = \mD(\pi^\ast\phi)\rv_{S^1} - \frac{m^2(m-1)}{4(m+n-1)}\sech^{\frac{4}{m}}\bigl(\frac{mR}{2}\bigr)\coth\bigl(\frac{mR}{2}\bigr)\Delta_{S^1}\phi - 3H_2\phi, \]
 where $\mD$ is the Dirichlet-to-Neumann operator~\eqref{eqn:dirichlet_to_neumann} and $\pi\colon\partial X_R\to S^1$ is the projection map.  It follows from the asymptotics of Lemma~\ref{lem:flat_computations} that the index of the restriction of $D\mF$ to $S^1$, and hence the index of $D\mF$ itself, tends to $\infty$ as $R\to\infty$.  Corollary~\ref{cor:bifurcation} then yields, the existence of the sequence $(R_j)_j$ of bifurcation points.
\end{proof}